\documentclass[11pt,letterpaper]{article}
\usepackage{fullpage}
\usepackage{amsmath,amssymb,amsthm}
\usepackage[ruled]{algorithm2e}
\usepackage{graphicx}
\usepackage[labelformat=simple]{subcaption}

\usepackage{hyperref}
\usepackage{xcolor}

\newcommand{\footnoteremember}[2]{
    \footnote{#2}%
    \newcounter{#1}
    \setcounter{#1}{\value{footnote}}
}
\newcommand{\footnoterecall}[1]{\footnotemark[\value{#1}]}

\newcommand{\minimize}{\mathop{\mathrm{minimize}{}}}

\newcommand{\argmin}{\mathop{\mathrm{arg\,min}{}}}

\newcommand{\dom}{\mathrm{dom\,}}

\newcommand{\rint}{\mathrm{rint\,}}
\newcommand{\reals}{\mathbb{R}}
\newcommand{\nats}{\mathbb{N}}

\newcommand{\incfac}{\rho}
\newcommand{\decfac}{\rho}
\newcommand{\zz}{\tilde{z}}
\newcommand{\DKL}{{D_\mathrm{KL}}}
\newcommand{\DIS}{{D_\mathrm{IS}}}

\newcommand{\gammain}{\gamma_{\mathrm{in}}}
\newcommand{\eigmax}{{\lambda_\mathrm{max}}}

\newtheorem{proposition}{Proposition}
\newtheorem{definition}{Definition}
\newtheorem{theorem}{Theorem}
\newtheorem{lemma}{Lemma}

\newtheorem{assumption}{Assumption}

\title{Accelerated Bregman proximal gradient methods for relatively smooth convex optimization}

\author{
Filip Hanzely\footnoteremember{KAUST}{
Division of Computer, Electrical and Mathematical Sciences and Engineering
(CEMSE), King Abdullah University of Science and Technology (KAUST), Thuwal, 
Kingdom of Saudi Arabia.
Emails: \texttt{filip.hanzely@kaust.edu.sa}, 
\texttt{peter.richtarik@kaust.edu.sa}}%
\footnote{Currently at Toyota Technological Institute at Chicago (TTIC).
Email: \texttt{filip@ttic.edu}}
\and 
Peter Richt{\'a}rik \footnoterecall{KAUST}
\footnote{School of Mathematics, The University of Edinburgh, Edinburgh, 
United Kingdom. }
\footnote{Moscow Institute of Physics and Technology, Dolgoprudny, Russia.} 
\and 
Lin Xiao \footnote{Work done while at Microsoft Research, Redmond, Washington, United States.
Email: \texttt{lin.xiao@gmail.com}}
}

\date{}

\begin{document}
\maketitle

\begin{abstract} 
We consider the problem of minimizing the sum of two convex functions: one is differentiable and relatively smooth with respect to a reference convex function, and the other can be nondifferentiable but simple to optimize.
We investigate a triangle scaling property of the Bregman distance generated by the reference convex function and present accelerated Bregman proximal gradient (ABPG) methods that attain an $O(k^{-\gamma})$ convergence rate, where $\gamma\in(0,2]$ is the \emph{triangle scaling exponent} (TSE) of the Bregman distance. 
For the Euclidean distance, we have $\gamma=2$ and recover the convergence rate of Nesterov's accelerated gradient methods. 
For non-Euclidean Bregman distances, the TSE can be much smaller (say $\gamma\leq 1$), but we show that a relaxed definition of \emph{intrinsic} TSE is always equal to~2. 
We exploit the intrinsic TSE to develop adaptive ABPG methods that converge much faster in practice. 
Although theoretical guarantees on a fast convergence rate seem to be out of reach in general, our methods obtain empirical $O(k^{-2})$ rates in 
numerical experiments on several applications
and provide posterior numerical certificates for the fast rates. 

\end{abstract} 

\section{Introduction}
Let $\reals^n$ be the $n$-dimensional real Euclidean space endowed with 
inner product $\langle x, y\rangle=\sum_{i=1}^n x^{(i)} y^{(i)}$ 
and the Euclidean norm $\|x\|=\sqrt{\langle x, x\rangle}$.
We consider optimization problems of the form 
\begin{equation}\label{eqn:composite-min}
\minimize_{x\in C} ~\bigl\{ F(x) := f(x) + \Psi(x) \bigr\}, 
\end{equation}
where $C$ is a closed convex set in~$\reals^n$, 
and~$f$ and~$\Psi$ are proper, closed convex functions.
We assume that~$f$ is differentiable on an open set that contains 
the relative interior of~$C$ (denoted as $\rint C$).
For the development of first-order methods, 
we also assume that~$C$ and~$\Psi$ are \emph{simple}, 
whose precise meaning will be explained in the context of specific algorithms.

First-order methods for solving~\eqref{eqn:composite-min} are often based
on the idea of minimizing a simple approximation of the objective~$F$ 
during each iteration. 
Specifically, in the \emph{proximal gradient method}, we start with an initial 
point $x_0\in\rint C$ and generate a sequence $x_k$ for $k=1,2,\ldots$ with
\begin{equation}\label{eqn:prox-grad}
x_{k+1} = \argmin_{x\in C} \Bigl\{ f(x_k) + \langle \nabla f(x_k), x-x_k\rangle
+\frac{L_k}{2}\|x-x_k\|^2 + \Psi(x) \Bigr\} ,
\end{equation}
where $L_k>0$ for all $k\geq 0$. 
Here, we use the gradient $\nabla f(x_k)$ to construct a local
quadratic approximation of $f$ around $x_k$ while leaving $\Psi$ untouched.
Our assumption that $C$ and $\Psi$ are simple means that the minimization
problem in~\eqref{eqn:prox-grad} can be solved efficiently, especially if it
admits a closed-form solution.

Assuming that $F$ is bounded below, 
convergence of the proximal gradient method can be established if 
$F(x_{k+1})\leq F(x_k)$ for all $k\in\nats$.
A sufficient condition for this to hold is that the quadratic approximation 
of~$f$ in~\eqref{eqn:prox-grad} is an upper approximation (majorization).
This is the basic idea behind many general methods for nonlinear optimization.
To this end, a common assumption is for the gradient of~$f$
to satisfy a uniform Lipschitz condition, i.e., 
there exists a constant $L_f$ such that
\begin{equation}\label{eqn:uniform-smooth}
\|\nabla f(x)-\nabla f(y)\| \leq L_f \|x-y\|, \quad \forall\, x,y\in\rint C.
\end{equation}
This smoothness assumption implies (see, e.g., \cite[Lemma~1.2.3]{Nesterov04book})
\begin{equation}\label{eqn:quadratic-upper-bound}
f(x)\leq f(y) + \langle\nabla f(y), x-y\rangle + \frac{L_f}{2}\|x-y\|^2,
\quad \forall\, x\in C, ~y\in\rint C.
\end{equation} 
Therefore, setting $L_k=L_f$ for all $k\in\nats$ 
ensures that the quadratic approximation of~$f$ in~\eqref{eqn:prox-grad}
is always an upper bound of~$f$, which implies $ F(x_{k+1})\leq F(x_k)$
for all $k\in\nats$. Moreover, it can be shown that the proximal gradient 
method enjoys an $O(k^{-1})$ convergence rate, i.e., 
\begin{equation}\label{eqn:PGM-rate}
     F(x_k) - F(x) \leq \frac{L_f}{k}\frac{\|x-x_0\|^2}{2}, 
    \quad \forall\, x\in C.
\end{equation}
See, e.g., \cite{BeckTeboulle09fista}, 
\cite{Nesterov13composite} and \cite[Chapter~10]{Beck17book}.
Under the same assumption, accelerated proximal gradient methods  
(\cite{Nesterov83,Nesterov04book,AuslenderTeboulle06,BeckTeboulle09fista,Tseng08,Nesterov13composite})
can achieve a faster $O(k^{-2})$ convergence rate:
\begin{equation}\label{eqn:accl-PGM-rate}
     F(x_k) - F(x) \leq \frac{4L_f}{(k+2)^2}\frac{\|x-x_0\|^2}{2}, 
    \quad \forall\, x\in C,
\end{equation}
which is optimal (up to a constant factor) for this class of convex 
optimization problems \cite{NemirovskiYudin83,Nesterov04book}.

\subsection{Optimization of relatively smooth functions}
\label{sec:relative-smooth}

While the uniform smoothness condition~\eqref{eqn:uniform-smooth} is 
central in the development and analysis of first-order methods, 
there are many applications where the objective function
does not have this property, despite being convex and differentiable. 
For example, in D-optimal experiment design 
(e.g., \cite{KieferWolfowitz59,Atwood69})
and Poisson inverse problems (e.g., \cite{Csiszar91,Bertero09}),
the objective functions involve the logarithm in the form of
log-determinant or relative entropy, 
whose gradients may blow up towards the boundary of the feasible region.
In order to develop efficient first-order algorithms for solving such
problems, the notion of \emph{relative smoothness} was introduced 
by several recent works 
\cite{BauschkeBolteTeboulle17,LuFreundNesterov18,ZhouLiangShen19}.

Let $h$ be a strictly convex function that is differentiable on an open set containing $\rint C$.
The Bregman distance associated with~$h$,
originated in \cite{Bregman67} and popularized by \cite{CensorLent1981iterative,CensorZenios1992proximal}, is defined as
\[
D_h(x,y) := h(x) - h(y) - \langle \nabla h(y), x-y\rangle, 
\quad \forall\, x\in \dom h, ~y\in \rint\dom h.
\]
\begin{definition}\label{def:relative-smooth}
The function~$f$ is called $L$-\emph{smooth relative} to~$h$ on~$C$ 
if there is an $L>0$ such that 
\begin{equation}\label{eqn:relative-smooth}
f(x)\leq f(y) + \langle\nabla f(y), x-y\rangle + L D_h(x,y),
\quad \forall\, x\in C, ~y\in\rint C.
\end{equation}
\end{definition}
As shown in \cite{BauschkeBolteTeboulle17} and \cite{LuFreundNesterov18}, 
this notion of relative smoothness is equivalent to the following statements:
\begin{itemize} \itemsep 0pt
    \item $Lh-f$ is a convex function on~$C$.
    \item If both $f$ and $h$ are twice differentiable, then 
        $\nabla^2 f(x) \preceq L\, \nabla^2 h(x)$ for all $x\in\rint C$.
    \item $\langle\nabla f(x)-\nabla f(y),x-y\rangle \leq L\langle 
        \nabla h(x)-\nabla h(y), x-y\rangle$ for all $x,y\in\rint C$.
\end{itemize}

The definition of relative smoothness in~\eqref{eqn:relative-smooth} gives an 
upper approximation of~$f$ that is similar to~\eqref{eqn:quadratic-upper-bound}.
In fact, \eqref{eqn:quadratic-upper-bound} is a special case 
of~\eqref{eqn:relative-smooth} with $h=(1/2)\|x\|^2$ and 
$D_h(x,y)=(1/2)\|x-y\|^2$.
Therefore it is natural to consider a more general algorithm by replacing
the squared Euclidean distance in~\eqref{eqn:prox-grad} with a 
Bregman distance:
\begin{equation}\label{eqn:BPG}
x_{k+1} = \argmin_{x\in C}\,\Bigl\{ f(x_k)+\langle\nabla f(x_k), x-x_k\rangle
+ L_k D_h(x,x_k) + \Psi(x) \Bigr\} .
\end{equation}
Here, our assumption that $C$ and $\Psi$ are simple means that the minimization
problem in~\eqref{eqn:BPG} can be solved efficiently. 
Similar to the proximal gradient method~\eqref{eqn:prox-grad}, this algorithm
can also be interpreted through operator splitting mechanism: 
it is the composition of a Bregman proximal step and a Bregman gradient step
(see details in \cite[Section~3.1]{BauschkeBolteTeboulle17}).
Therefore, it is called the \emph{Bregman proximal gradient} (BPG) method
\cite{Teboulle18simplied}.

Under the relative smoothness condition~\eqref{eqn:relative-smooth}, setting
$L_k=L$ ensures that the function being minimized in~\eqref{eqn:BPG} is a
majorization of~$ F$, which implies $ F(x_{k+1})\leq F(x_k)$ for all
$k\in\nats$.
It was first shown in~\cite{BDX:11} (for the case $\Psi\equiv 0$) 
that the BGD method has a $O(k^{-1})$ convergence rate:
\[
     F(x_k) - F(x) \leq \frac{L}{k}D_h(x,x_0), 
    \quad \forall\, x\in \dom h.
\]
This is a generalization of~\eqref{eqn:PGM-rate}.
The same convergence rate for the general case (with nontrivial $\Psi$)
is obtained in \cite{BauschkeBolteTeboulle17}, where the authors also 
discussed the effect of a symmetry measure for the Bregman distance.
Similar results are also obtained in \cite{LuFreundNesterov18} and
\cite{ZhouLiangShen19}. In addition, \cite{LuFreundNesterov18} 
introduced the notion of relative strong convexity and obtained linear
convergence of the BPG method when both relative smoothness and 
relative strong convexity hold.
More recently, \cite{hanzely2018fastest} studied stochastic gradient descent 
and randomized coordinate descent methods in the relatively smooth setting,
and \cite{Lu17rel_continuity} extended this framework to
minimize relatively continuous convex functions.

A natural question is whether the $O(k^{-1})$ rate can be improved with first-order methods under the relative smoothness assumption, especially whether the accelerated $O(k^{-2})$ rate can be achieved
\cite{LuFreundNesterov18,Teboulle18simplied}. 
Very recently, it is shown by Dragomir et al.~\cite{Dragomir2019} that the $O(k^{-1})$ rate is optimal for the class of relatively smooth functions, thus cannot be improved in general.
However, we note that the class of relatively smooth functions is very broad, containing differentiable functions whose gradients has arbitrarily large Lipschitz constants. Indeed, the worst-case function constructed in \cite{Dragomir2019} to prove the lower bound is obtained by smoothing a nonsmooth function, which demonstrate pathological nonsmooth behavior.
This is in sharp contrast to the situation under the uniform Lipschitz assumption, which uses a \emph{fixed} quadratic function as the relatively smooth measure.

Ideally, it would be most informative to derive both upper and lower bounds on the convergence rate of first-order methods for every \emph{fixed} function~$h$ in the relatively smooth setting, or at least for the popular ones that are frequently encountered in application (such as the KL divergence).
It is plausible that the achievable convergence rates for particular functions~$h$ (more likely particular combinations of~$f$ and~$h$) can be better than $O(k^{-1})$ in theory or at least in practice.
A full spectrum investigation is beyond the scope of this paper. Instead, we study a structural property of general Bregman divergences called \emph{triangle scaling} and develop adaptive first-order methods that, although without a priori guarantee, often demonstrate the $O(k^{-2})$ convergence rate empirically in many applications. Moreover, these methods produce simple \emph{numerical certificates} of the fast rates whenever they happen.

\subsection{Contributions and outline}

First, in Section~\ref{sec:triangle-scaling}, we study a triangle-scaling
property for general Bregman distances and define the triangle-scaling exponent (TSE)~$\gamma>0$, which is key in characterizing the convergence rates of first-order methods in the relatively smooth setting.
We estimate the value of~$\gamma$ for several Bregman distances that appear 
frequently in applications. 
Moreover, we define an intrinsic triangle-scaling exponent $\gammain$ and show that $\gammain=2$ for all~$h$ that is twice continuously differentiable.

In Section~\ref{sec:ABPG}, we propose a basic accelerated Bregman proximal gradient (ABPG) method that attains an $O(k^{-\gamma})$ convergence rate, 
where $\gamma\leq 2$ is the TSE of the Bregman divergence.
More specifically, under the assumption~\eqref{eqn:relative-smooth},
the basic ABPG method produces a sequence $\{x_k\}$ satisfying
\begin{equation}\label{eqn:ABPG-rate}
 F(x_k)- F(x)\leq\left(\frac{\gamma}{k+\gamma}\right)^\gamma L D_h(x,x_0), 
\quad \forall\, x\in \dom h.
\end{equation}
The exact value of~$\gamma$ depends on a \emph{triangle scaling property}
of the Bregman distance. For $D_h(x,y)=(1/2)\|x-y\|^2$, 
we have $\gamma=2$ and $L=L_f$, hence the result in~\eqref{eqn:ABPG-rate} 
recovers that in~\eqref{eqn:accl-PGM-rate}.
We also give an adaptive variant that can automatically search for the
largest possible~$\gamma$ for which the convergence rate
in~\eqref{eqn:ABPG-rate} holds for finite~$k$ even though~$\gamma$ 
is larger than the TSE.

In Section~\ref{sec:ABPG-gain}, we develop an adaptive ABPG method that 
automatically adjust an additional gain factor in order to work with the 
intrinsic TSE $\gammain=2$.
If the geometric mean of the gains obtained at all the iterations up to~$k$ 
is a small constant, say $O(1)$, then they constitute numerical certificates that the algorithm has enjoyed an empirical $O(k^{-2})$ convergence rate.

In Section~\ref{sec:ABDA}, we present an accelerated Bregman dual-averaging 
algorithm that has similar convergence rates as the basic ABPG method, 
but omit discussions of its adaptive variants.

Finally, in Section~\ref{sec:experiments}, we present numerical experiments
with three applications: the D-optimal experiment design problem, 
a Poisson linear inverse problem, and relative-entropy nonnegative regression.
In all experiments, the ABPG methods, especially the adaptive variants, 
demonstrate superior performance compared with the BPG method.
Moreover, we obtain numerical certificates for the empirical $O(k^{-2})$ rate 
in all our experiments.

\paragraph{Related work.}
The relative smoothness condition directly extends the upper approximation
property~\eqref{eqn:quadratic-upper-bound} with more general Bregman distances.
Nesterov \cite{Nesterov15universal} took an alternative approach by 
extending the Lipschitz condition~\eqref{eqn:uniform-smooth}.
Specifically, he considered functions with H\"older continuous gradients 
with a parameter $\nu\in[0,1]$:
\[
\|\nabla f(x)-\nabla f(y)\|_* \leq L_\nu \|x-y\|^\nu, 
\quad x,y\in C,
\]
and obtained $O(k^{-(1+\nu)/2})$ rate with a universal gradient method
and $O(k^{-(1+3\nu)/2})$ rate with accelerated schemes. 
These methods are called ``universal'' because they do not assume the knowledge
of~$\nu$ and automatically ensure the best possible rate of convergence.
The accelerated $O(k^{-(1+3\nu)/2})$ rate 
interpolates between $O(k^{-1/2})$ and $O(k^{-2})$ with $\nu\in[0,1]$.
There seems to be no simple connection or correspondence between 
the H\"older smoothness property and the combination of relative smoothness
and the triangle scaling property studied in this paper.

Gutman and Pe{\~n}a \cite{GutmanPena2018}
studied iteration complexity of first-order methods using a general framework of perturbed Fenchel duality.
Their framework provides alternative derivations of the convergence rates of Bregman proximal gradient methods under the relative smooth setting and the ones under H{\"o}lder continuity assumption.

\paragraph{Technical assumptions.}
Development and analysis of optimization methods in the relatively smooth
setting require some delicate assumptions in order to cover many interesting 
applications without loss of rigor.
Here we adopt the same assumptions made in \cite{BauschkeBolteTeboulle17}
regarding problem~\eqref{eqn:composite-min}.

\begin{assumption}\label{asmp:all}
Suppose that $C$ is a closed convex set in $\reals^n$ and $h:\reals^n\to(-\infty,+\infty]$ is strictly convex and differentiable on an open set containing $\rint C$.
Moreover,
\begin{enumerate} \itemsep 0pt
    \item 
        $h$ is of Legendre type
        \cite[Section~26]{Rockafellar70book}. In other words, it is essentially 
        smooth and strictly convex in $\rint\dom h$. 
        Essential smoothness means that it is differentiable and 
        $\|\nabla h(x_k)\|\to\infty$ for every sequence 
        $\{x_k\}_{k\in\nats}$ converging to a boundary point of $\dom h$.
    \item $f:\reals^n \to(-\infty,\infty]$ is a proper and closed convex 
        function, and it is differentiable on $\rint C$.
    \item $\Psi:\reals^n \to(-\infty,\infty]$ is a proper and closed convex 
        function, and $\dom\Psi \cap \rint\dom h \neq \emptyset$.
    \item $\inf_{x\in C}\{f(x)+\Psi(x)\}>-\infty$, i.e., 
        problem~\eqref{eqn:composite-min} is bounded below.
    \item The BPG step~\eqref{eqn:BPG} is well posed, 
        meaning that $x_{k+1}$ is unique and belongs to $\rint\dom h$.
\end{enumerate}
\end{assumption}
Sufficient conditions for the well-posedness of~\eqref{eqn:BPG} are given in
\cite[Lemma~2]{BauschkeBolteTeboulle17}. The same conditions also ensure
that our proposed accelerated methods are well-posed.

\section{Triangle scaling of Bregman distance}
\label{sec:triangle-scaling}

In this section, we define the \emph{triangle scaling property} for 
Bregman distances and discuss two different notions of 
\emph{triangle scaling exponent} (TSE).

\begin{definition}\label{def:triangle-scaling}
Let $h$ be a convex function that is differentiable on $\rint\dom h$.
The Bregman distance $D_h$ has the \emph{triangle scaling property} 
if there is a constant $\gamma>0$ 
such that for all $x,z,\zz\in\rint\dom h$,
\begin{equation}\label{eqn:triangle-scaling}
D_h\big((1-\theta)x+\theta z,\;(1-\theta)x+\theta \zz\big)~\leq~
\theta^\gamma D_h(z,\zz),\qquad \forall\, \theta\in[0,1].
\end{equation}
We call $\gamma$ a uniform \emph{triangle scaling exponent} (TSE) of $D_h$.
\end{definition}

\begin{figure}[t]
\centering
\includegraphics[width=0.4\textwidth]{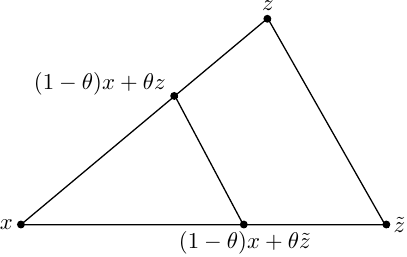}
\vspace{2ex}
\caption{Illustration of different points in the triangle scaling property.}
\label{fig:triangle}
\end{figure}

Figure~\ref{fig:triangle} gives a geometric illustration of the points 
involved in the above definition.

\bigskip 

If $D_h(x,y)$ is jointly convex in $(x,y)$, then the 
inequality~\eqref{eqn:triangle-scaling} holds with $\gamma=1$ 
because
\[
D_h\big((1-\theta)x+\theta z,\;(1-\theta)x+\theta \zz\big)
\leq (1-\theta)D_h(x,x) + \theta D_h(z,\zz) =\theta D_h(z,\zz).
\]
Therefore it is useful to study jointly convex Bregman distances.
Suppose $h:\reals\to(-\infty,\infty]$ is strictly convex and 
twice continuously differentiable on an open interval in~$\reals$.
Let $h''$ denotes the second derivative of~$h$.
It was shown in \cite{BauschkeBorwein2001} that 
the Bregman distance $D_h(\cdot,\cdot)$ is jointly convex 
if and only if $1/h''$ is concave.
This result applies directly to separable functions which can be written as
$h(x)=\sum_{i=1}^n h_i(x^{(i)})$.
If $1/h''_i$ is concave for each~$i=1,\ldots,n$, then we conclude that 
$D_h$ has a uniform TSE of at least $1$.
Below are some specific examples:
\begin{itemize}
\item \emph{The squared Euclidean distance.}
Let $h(x)=(1/2)\|x\|_2^2$ and $D_h(x,y) = (1/2)\|x-y\|_2^2$.
Obviously, here $D_h$ is jointly convex in its two arguments. 
But it is also easy to see that
\[
\frac{1}{2}\left\|(1-\theta)x+\theta z - \big((1-\theta)x+\theta \zz\big) 
\right\|_2^2
=\frac{1}{2}\|\theta(z-\zz)\|_2^2 = \theta^2\frac{1}{2}\|z-\zz\|_2^2.
\]
Therefore the squared Euclidean distance has a uniform TSE $\gamma=2$, which 
is much larger than~$1$ obtained by following the jointly convex argument.
\item 
\emph{Bregman distance induced by strongly convex and smooth functions.}
If~$h$ is $\mu$-strongly convex and $L$-smooth over its domain, 
then the inequality~\eqref{eqn:triangle-scaling} would hold with $\gamma=2$ 
if the right-hand side is multiplied by an additional factor~$G=L/\mu$,
which is the condition number of~$h$.
We will prove this fact in Section~\ref{sec:bound-ts-gain}.
\item \emph{The generalized Kullback-Leibler (KL) divergence.}
Let~$h$ be the negative Boltzmann-Shannon entropy:
$h(x)=\sum_{i=1}^n x^{(i)}\log x^{(i)}$ defined on $\reals^n_+$. 
The Bregman distance associated with~$h$ is
\begin{equation}\label{eqn:KL-divergence}
D_\mathrm{KL}(x,y) 
=\sum_{i=1}^n \biggl( x^{(i)}\log\biggl(\frac{x^{(i)}}{y^{(i)}}\biggr) 
- x^{(i)} + y^{(i)}\biggr).
\end{equation}
Since $1/h''_i=x^{(i)}$ is linear thus concave for each~$i$, we conclude that 
$D_\mathrm{KL}(x,y)$ is jointly convex in $(x,y)$, which implies that
it has a uniform TSE $\gamma=1$.
\item \emph{The Itakura-Saito (IS) distance.}
The IS distance is the Bregman distance generated by Burg's entropy 
$h(x)=\sum_{i=1}^n -\log(x^{(i)})$ with $\dom h=\reals^n_{++}$:
\begin{equation}\label{eqn:IS-distance}
D_\mathrm{IS}(x,y) = \sum_{i=1}^n \biggl(
-\log\biggl(\frac{x^{(i)}}{y^{(i)}}\biggr)+\frac{x^{(i)}}{y^{(i)}} - 1\biggr).
\end{equation}
Since $1/h''_i=(x^{(i)})^2$ is not concave, we conclude that
$D_\mathrm{IS}(\cdot,\cdot)$ is not jointly convex.
Hence if it has a uniform TSE, then it is likely to be less than~$1$.
In fact, it can be easily checked numerically that any $\gamma>0.5$ 
is not a uniform TSE for $D_\mathrm{IS}$ when $G=1$.
\item 
\emph{Bregman distance based on polynomial kernels.}
Reference functions of the form $h(x)=(1/p)\|x\|^p$ for some $p\geq 2$
recently attracted lots of attention following Nesterov's work on tensor methods in convex optimization \cite{Nesterov2019tensor}.
In general, the global TSEs for the induced Bregman divergence can be less than~$1$ for $p>2$. However, the modified reference function $h(x)=(1/2)\|x\|^2 + (1/p)\|x\|^p$ for $p\geq 4$ has $\gamma>1$, or $\gamma=2$ with an additional factor on the right-hand side of~\eqref{eqn:triangle-scaling}, over a bounded domain. 
We will give detailed analysis for the case $p=4$ in Section~\ref{sec:bound-ts-gain}, after introducing a relaxed version of TSE.
\end{itemize}

We observe that the largest uniform TSEs are quite different for the 
Bregman distances listed above. 
An important question is: Are these differences essential such that they
lead to different convergence rates if different Bregman distances are used
in an accelerated algorithm?
It would be ideal to derive an intrinsic characterization that is common for 
most Bregman distances and essential for convergence analysis of 
accelerated algorithms.


\subsection{The intrinsic triangle-scaling exponent}
\label{sec:intrinsic-TSE}

For any fixed triple $\{x,z,\zz\}\subset\rint\dom h$, 
we consider a relaxed version of triangle scaling:
\begin{equation}\label{eqn:gain-triangle-scaling}
D_h\big((1-\theta)x+\theta z,\;(1-\theta)x+\theta \zz\big)
~\leq~G(x,z,\zz)\, \theta^\gamma D_h(z,\zz),\qquad \forall\, \theta\in[0,1],
\end{equation}
where $G(x,z,\zz)$ depends on the triple $\{x,z,\zz\}$ 
but does not depend on~$\theta$.
The intrinsic TSE of $D_h$, denoted $\gammain$, is the largest~$\gamma$ 
such that~\eqref{eqn:gain-triangle-scaling} holds with some finite $G(x,z,\zz)$ 
for all triples $\{x,z,\zz\}\subset\rint\dom h$.

Notice that when $\theta$ is bounded away from~$0$, we can always find 
sufficiently large $G(x,z,\zz)$ to make the inequality 
in~\eqref{eqn:gain-triangle-scaling} hold with any value of $\gamma$. 
Therefore, the intrinsic TSE is determined only by the asymptotic behavior of 
$D_h\big((1-\theta)x+\theta z,(1-\theta)x+\theta \zz\big)$ when $\theta\to 0$.
A more precise definition is as follows.

\begin{definition}
The intrinsic TSE of $D_h$, denoted $\gammain$, is the largest~$\gamma$ such that for all $x,z,\zz\in \rint\dom h$,
\[
\limsup_{\theta\to 0} 
\frac{D_h\big((1-\theta)x+\theta z,(1-\theta)x+\theta \zz\big)}{\theta^\gamma}
~<~ \infty .
\]
\end{definition}
We show that a broad family of Bregman distances
share the same intrinsic TSE $\gammain=2$. 

\begin{proposition}\label{prop:intrinsic-TSE}
If $h$ is convex and twice continuously differentiable on $\rint\dom h$, 
then the intrinsic TSE of the Bregman distance $D_h$ is $2$.
Specifically, for any $\{x,z,\zz\}\subset\rint\dom h$, we have
\begin{equation}\label{eqn:limit-square-ratio}
\lim_{\theta\to 0} 
\frac{D_h\big((1-\theta)x+\theta z,\;(1-\theta)x+\theta \zz\big)}{\theta^2}
=\frac{1}{2} \bigl\langle \nabla^2 h(x) (z-\zz), z-\zz\bigr\rangle.
\end{equation}
\end{proposition}
\begin{proof}
Consider the limit in~\eqref{eqn:limit-square-ratio}, since both the numerator 
$D_h\big((1-\theta)x+\theta z,\;(1-\theta)x+\theta \zz\big)$
and the denominator $\theta^2$ converge to zero as $\theta\to 0$,
we apply L'Hospital's rule. 
First, by definition of the Bregman distance, we have
\begin{eqnarray*}
&& D_h\big((1-\theta)x+\theta z,\;(1-\theta)x+\theta \zz\big) \\
&=& D_h\big(x+\theta (z-x),\;x+\theta (\zz-x) \big) \\
&=& h\big(x+\theta (z-x)\bigr)- h\big(x+\theta (\zz-x)\big)
-\bigl\langle\nabla h\big(x+\theta (\zz-x)\big), \theta(z-\zz)\bigr\rangle .
\end{eqnarray*}
The derivative of $D_h\big((1-\theta)x+\theta z,\;(1-\theta)x+\theta \zz\big)$
with respect to $\theta$ is
\[
\frac{d}{d\theta} D_h\big((1-\theta)x+\theta z,\;(1-\theta)x+\theta \zz\big) 
~=~ A(\theta) 
-\bigl\langle\nabla^2 h(x+\theta(\zz-x))(\zz-x), \theta(z-\zz)\bigr\rangle,
\]
where
\[
A(\theta)
= \bigl\langle\nabla h(x+\theta(z-x)),z-x\bigr\rangle
-\bigl\langle\nabla h(x+\theta(\zz-x)),\zz-x\bigr\rangle 
-\bigl\langle\nabla h(x+\theta(\zz-x)),z-\zz\bigr\rangle.
\]
Therefore,
\begin{align}
\lim_{\theta\to 0} 
\frac{D_h\big((1-\theta)x+\theta z,(1-\theta)x+\theta \zz\big)}{\theta^2}
&=\lim_{\theta\to 0} \frac{A(\theta) -\bigl\langle 
\nabla^2 h\bigl(x+\theta(\zz-x)\bigr)(\zz-x),\theta(z-\zz)\bigr\rangle}{2\theta}
\nonumber \\
&=\lim_{\theta\to 0} \frac{A(\theta)}{2\theta} 
-\lim_{\theta\to 0} \frac{ 
\bigl\langle\nabla^2 h\bigl(x+\theta(\zz-x)\bigr)(\zz-x), z-\zz\bigr\rangle}{2}
\nonumber \\
&=\lim_{\theta\to 0} \frac{A(\theta)}{2\theta} 
- \frac{1}{2} \bigl\langle\nabla^2 h(x)(\zz-x), z-\zz\bigr\rangle.
\label{eqn:second-order-limit}
\end{align}
Notice that
\[
\lim_{\theta\to 0} A(\theta) 
= \bigl\langle\nabla h(x),z-x\bigr\rangle
-\bigl\langle\nabla h(x),\zz-x\bigr\rangle 
-\bigl\langle\nabla h(x),z-\zz\bigr\rangle = 0,
\]
so we apply L'Hospital's rule again:
\[
\lim_{\theta\to 0} \frac{A(\theta)}{2\theta}
= \frac{\bigl\langle\nabla^2 h(x)(z-x),z-x\bigr\rangle
-\bigl\langle\nabla^2 h(x)(\zz-x),\zz-x\bigr\rangle 
-\bigl\langle\nabla^2 h(x)(\zz-x),z-\zz\bigr\rangle}{2}.
\]
Plugging the last equality into~\eqref{eqn:second-order-limit} 
and after some simple algebra, we arrive at~\eqref{eqn:limit-square-ratio}.
\end{proof}

According to Proposition~\ref{prop:intrinsic-TSE}, the examples we considered
earlier, including the generalized KL-divergence and
the IS-distance, share the same intrinsic TSE $\gammain=2$.
Proposition~\ref{prop:intrinsic-TSE} also implies that the largest 
uniform TSE cannot exceed~$2$.

\subsection{Bounding the triangle-scaling gain}
\label{sec:bound-ts-gain}

In our analysis of accelerated algorithms in the relatively smooth setting, 
it is crucial to bound the triangle scaling gain $G(x,z,\zz)$.
Here we derive a general bound based on the relative scaling of the Hessians of~$h$ at different points. 
First, by the second-order Taylor expansion (mean value theorem), we have
\[
D_h(x,y) = h(x) - h(y) - \langle \nabla h(y), x-y\rangle 
=\frac{1}{2}(x-y)^T \nabla^2 h(w) (x-y),
\]
where $w=x+t(y-x)$ for some $t\in[0,1]$, which we denote as $w\in[x, y]$.
Consequently, if we define
\[
G_{\theta}(x, z, \zz) := \frac{D_h\bigl((1-\theta)x+\theta z, (1-\theta)x+\theta \zz\bigr)}{\theta^2 \cdot D_h(z,\zz)},
\]
then for some $u\in[(1-\theta)x+\theta z, (1-\theta)x+\theta \zz]$ and $v\in[z,\zz]$, we have
\[
G_{\theta}(x, z, \zz) 
=\frac{\frac{1}{2}\theta^2 (z-\zz)\nabla^2 h(u)(z-\zz)}{\theta^2\cdot \frac{1}{2}(z-\zz)^T\nabla^2 h(v)(z-\zz)}
=\frac{(z-\zz)\nabla^2 h(u)(z-\zz)}{(z-\zz)^T\nabla^2 h(v)(z-\zz)}.
\]
The last expression does not depend on~$\theta$ explicitly, but through $u\in[(1-\theta)x+\theta z, (1-\theta)x+\theta \zz]$. 
If $\theta\to 0$, then we have $u\to x$.

In general, let's assume $u,v\in\rint\dom h$ and $\nabla^2 h(v)$ is non-singular. Then we have
\begin{equation}\label{eqn:gain-eigen-bd}
G_{\theta}(x,z,\zz) \leq \eigmax\left(\nabla^2 h(v)^{-1/2} \nabla^2 h(u) \nabla^2 h(v)^{-1/2}\right),
\end{equation}
where $\eigmax(\cdot)$ denotes the maximum eigenvalue of a positive semidefinite matrix.
Therefore, the triangle-scaling gain is bounded by how close the two Hessians $\nabla^2 h(u)$ and $\nabla^2 h(v)$ are.
Since convex quadratic functions of the form $h(x)=(1/2)x^T A x + b^T x + c$
has constant Hessian $\nabla^2 h(x)=A$, the triangle scaling gain is always~$1$, independent of~$\theta$.

More generally, if $h$ is strongly convex and smooth (the eigenvalues of its Hessian have positive lower and upper bounds), then the triangle scaling gain is bounded by its condition number, i.e., the ratio between the upper and lower bounds on the Hessian eigenvalues.
Otherwise, the gain can be unbounded without any proximity assumption on the three points $(x,z,\zz)$. 

Next we consider the polynomial reference function
$h(x)=(1/4)\|x\|^4$, which does not have bounded Hessian.
In this case, we have $\nabla h(x)=\|x\|^{2} x$ and 
$\nabla^2 h(x)=\|x\|^{2}\cdot I + 2 xx^T$, where~$I$ is the identity matrix.
Clearly $\|x\|^2\cdot I \preceq \nabla^2 h(x) \preceq 3\|x\|^2\cdot I$.
According to~\eqref{eqn:gain-eigen-bd}, we have
\[
G_{\theta}(x,z,\zz)\leq 3\frac{\|u\|^2}{\|v\|^2},
\]
for some $u\in[(1-\theta)x+\theta z, (1-\theta)x+\theta \zz]$ and $v\in[z,\zz]$.
Therefore, it is not hard to construct examples with $v\approx 0$ thus the triangle scaling gain can be unbounded, even if the points $(x,z,\zz)$ are close in a small neighborhood (near the origin).

As a simple fix, we consider $h(x) = (1/2)\|x\|^2+(1/4)\|x\|^4$, 
whose Hessian is $\nabla^2 h(x)=(1+\|x\|^2) I + 2 xx^T$ and it satisfies 
$(1+\|x\|^2) I \preceq \nabla^2 h(x) \preceq (1+3\|x\|^2) I$.
Therefore, according to~\eqref{eqn:gain-eigen-bd}, 
\[
G_{\theta}(x,z,\zz)\leq \frac{1+3\|u\|^2}{1+\|v\|^2}.
\]
In this case, it is clear that $G_{\theta}(x,z,\zz)$ (associated with $\gamma_{\mathrm{in}}=2$) is always bounded if the three points $(x,z,\zz)$ are bounded, even if $\|v\|=0$. 
In particular, if the domain of consideration, $\dom\Psi$, is bounded with radius~$R$ from the origin, then 
$G_{\theta}(x,z,\zz)\leq 1+3R^2$.

\section{Accelerated Bregman proximal gradient method}
\label{sec:ABPG}
In this section, we present an accelerated Bregman proximal gradient (ABPG)
method for solving problem~\eqref{eqn:composite-min}, 
and analyze its convergence rate under the uniform triangle-scaling property. 
Adaptive variants based on the intrinsic TSE are developed in 
Section~\ref{sec:ABPG-gain}.

To simplify notation, we define a lower approximation 
of~$ F(x)=f(x)+\Psi(x)$ by linearizing~$f$ at a given point~$y$:
\[
    \ell(x|y) := f(y) + \langle\nabla f(y), x-y\rangle + \Psi(x).
\]
If $f$ is $L$-smooth relative to~$h$ (Definition~\ref{def:relative-smooth}), 
then we have both a lower and an upper approximation:
\begin{equation}\label{eqn:sanwitch}
\ell(x|y) ~\leq~  F(x) ~\leq~ \ell(x|y) + L D_h(x,y).
\end{equation}

Algorithm~\ref{alg:ABPG} describes the ABPG method.
Its input parameters include a uniform TSE $\gamma$ of $D_h$ 
and an initial point $x_0\in\rint C$. 
The sequence $\{\theta_k\}_{k\in\nats}$ in Algorithm~\ref{alg:ABPG}
satisfies $0<\theta_k\leq 1$ and
\begin{equation}\label{eqn:theta-inequality}
\frac{1-\theta_{k+1}}{\theta_{k+1}^\gamma} \leq \frac{1}{\theta_k^\gamma},
\qquad \forall\, k\geq 0. 
\end{equation}
When $\gamma=2$ and $\Psi\equiv 0$, Algorithm~\ref{alg:ABPG} reduces to 
the IGA (improved interior gradient algorithm) method
in \cite{AuslenderTeboulle06}, which is an 
extension of Nesterov's accelerated gradient method in \cite{Nesterov88} 
to the Bregman proximal setting.
It was shown in \cite{AuslenderTeboulle06} that the IGA method attains
$O(k^{-2})$ rate of convergence under the uniform Lipschitz 
condition~\eqref{eqn:uniform-smooth}. 
In this paper, we consider the general case $\gamma\in[1,2]$ under the
much weaker relatively smooth condition. 

\begin{algorithm}[t]
\caption{Accelerated Bregman proximal gradient (ABPG) method}
\label{alg:ABPG}
\linespread{1.2}\selectfont
\DontPrintSemicolon
\textbf{input:} initial point $x_0 \in \rint C$ and $\gamma\geq 1$. \;
initialize: $z_0=x_0$ and $\theta_0=1$. \;
\For{$k=0,1,2,\dots $}{
\nl\label{lnl:ABPG-yk}%
$y_k = (1-\theta_k) x_k + \theta_k z_k $ \; 
\nl\label{lnl:ABPG-zk}%
$z_{k+1} = \argmin_{z \in C} \bigl\{ \ell(z|y_k) 
 + \theta_k^{\gamma-1} L D_h(z,z_k)\bigr\} $ \; 
\nl\label{lnl:ABPG-xk}%
$x_{k+1} = (1-\theta_k) x_k + \theta_k z_{k+1} $ \; 
\nl choose $\theta_{k+1}\in(0,1]$ such that
$\frac{1-\theta_{k+1}}{\theta_{k+1}^\gamma} \leq \frac{1}{\theta_k^\gamma}$\;
}
\end{algorithm}

Using the definition of $\ell(\cdot|\cdot)$, 
line~\ref{lnl:ABPG-zk} in Algorithm~\ref{alg:ABPG} can be written as
\begin{equation}\label{eqn:ABPG-zk}
z_{k+1} = \argmin_{x\in C}\,\Bigl\{ f(y_k)+\langle\nabla f(y_k), x-y_k\rangle
+ \theta_k^{\gamma-1} L D_h(x,z_k) + \Psi(x) \Bigr\} ,
\end{equation}
which is very similar to the BPG step~\eqref{eqn:BPG}. 
Here the function~$f$ is linearized around~$y_k$ 
but the Bregman distance is measured from a different point~$z_k$.
Therefore it does not fit into the framework of majorization and 
the sequence $ F(x_k)$ may not be monotone decreasing.
However, the upper bound in~\eqref{eqn:sanwitch} is still crucial 
to ensure convergence of the algorithm.
Under the same assumption that the BPG step is well-posed
(Assumption~\ref{asmp:all}.5), the ABPG method is also well-posed,
meaning that $z_{k+1}\in\rint C$ always and it is unique.

\subsection{Convergence analysis of ABPG}
\label{sec:ABPG-analysis}

We show that the ABPG method converges with a sublinear rate of
$O(k^{-\gamma})$.
First, we state a basic property of optimization with Bregman distance
\cite[Lemma~3.2]{ChenTeboulle93}.
\begin{lemma}\label{lem:basic-lemma-D}
For any closed convex function $\varphi:\reals^n\to(-\infty,\infty]$ and
any $z\in\rint \dom h$, if 
\[
    z_+ = \argmin_{x\in C}\, \bigl\{\varphi(x) + D_h(x,z)\bigr\}
\]
and $h$ is differentiable at $z_+$, then
\[
    \varphi(x) + D_h(x,z) \geq \varphi(z_+) + D_h(z_+,z) + D_h(x,z_+),
    \quad \forall x\in\dom h.
\]
\end{lemma}

The following lemma establishes a relationship between the two 
consecutive steps of Algorithm~\ref{alg:ABPG}.
It is an extension of Proposition~1 in \cite{Tseng08}
, which uses $\gamma=2$ under the assumption~\eqref{eqn:uniform-smooth}.

\begin{lemma}\label{lem:ABPG-onestep}
Suppose Assumption~\ref{asmp:all} holds, 
$f$ is $L$-smooth relative to~$h$ on~$C$, 
and $\gamma$ is a uniform TSE of~$D_h$.
For any $x\in\dom h$, the sequences generated by Algorithm~\ref{alg:ABPG} 
satisfy, for all $k\geq 0$, 
\begin{equation}\label{eqn:ABPG-onestep}
\frac{1-\theta_{k+1}}{\theta_{k+1}^\gamma} \bigl( F(x_{k+1}) -  F(x)\bigr)
    + L D_h(x, z_{k+1}) 
~\leq~ \frac{1-\theta_k}{\theta_k^\gamma} \bigl( F(x_k) -  F(x) \bigr)
    + L D_h(x, z_k) .
\end{equation}
\end{lemma}

\begin{proof}
First, using the upper approximation in~\eqref{eqn:sanwitch} and 
line~\ref{lnl:ABPG-yk} and line~\ref{lnl:ABPG-xk} in Algorithm~\ref{alg:ABPG},
we have
\begin{eqnarray}
 F(x_{k+1})
&\leq& \ell(x_{k+1}|y_k) + L D_h(x_{k+1}, y_k) \nonumber \\
&=& \ell(x_{k+1}|y_k) + L D_h\bigl((1-\theta_k)x_k+\theta_k z_{k+1},
    (1-\theta_k)x_k+\theta_k z_k\bigr)  \nonumber \\
&\leq& \ell(x_{k+1}|y_k) + \theta_k^\gamma L D_h(z_{k+1}, z_k),
\label{eqn:ABPG-triangle}
\end{eqnarray}
where in the last inequality
we used the triangle-scaling property~\eqref{eqn:triangle-scaling}.
Using $x_{k+1}=(1-\theta_k)x_k+\theta_k z_{k+1}$ and
convexity of $\ell(\cdot|y_k)$, we have
\begin{eqnarray}
 F(x_{k+1})
&\leq& (1-\theta_k)\ell(x_k|y_k) + \theta_k \ell(z_{k+1}|y_k) 
    + \theta_k^\gamma L D_h(z_{k+1}, z_k)  \label{eqn:dual-avg-start}\\
&=& (1-\theta_k)\ell(x_k|y_k) + \theta_k \left( \ell(z_{k+1}|y_k) 
    + \theta_k^{\gamma-1} L D_h(z_{k+1}, z_k) \right) \nonumber .
\end{eqnarray}
Now applying Lemma~\ref{lem:basic-lemma-D} with 
$\varphi(x)=\ell(x|y_k)/(\theta^{\gamma-1}L)$ yields, for any $x\in\dom h$,
\[
\ell(z_{k+1}|y_k) + \theta_k^{\gamma-1} L D_h(z_{k+1}, z_k)
~\leq~ \ell(x|y_k) + \theta_k^{\gamma-1} L D_h(x, z_k)  
- \theta_k^{\gamma-1} L D_h(x, z_{k+1}) .
\]
Hence
\begin{eqnarray*}
 F(x_{k+1})
&\leq& (1-\theta_k)\ell(x_k|y_k) + \theta_k \left( \ell(x|y_k) 
    + \theta_k^{\gamma-1} L D_h(x, z_k)  
    - \theta_k^{\gamma-1} L D_h(x, z_{k+1}) \right) \\
&=& (1-\theta_k)\ell(x_k|y_k) + \theta_k \ell(x|y_k) 
    + \theta_k^{\gamma} \bigl( L D_h(x, z_k) - L D_h(x, z_{k+1}) \bigr) \\
&\leq& (1-\theta_k) F(x_k) + \theta_k  F(x) 
    + \theta_k^{\gamma} \bigl( L D_h(x, z_k) - L D_h(x, z_{k+1}) \bigr) ,
\end{eqnarray*}
where in the last inequality we used the lower bound in~\eqref{eqn:sanwitch}.
Subtracting $ F(x)$ from both sides of the inequality above, we obtain
\[
 F(x_{k+1}) -  F(x) ~\leq~ (1-\theta_k)\bigl( F(x_k) -  F(x) \bigr)
+ \theta_k^{\gamma} \bigl( L D_h(x, z_k) - L D_h(x, z_{k+1}) \bigr).
\]
Dividing both sides by $\theta_k^\gamma$ and rearranging terms yield
\begin{equation}\label{eqn:last-recursion}
\frac{1}{\theta_k^\gamma} \bigl( F(x_{k+1}) -  F(x)\bigr)
    + L D_h(x, z_{k+1}) 
~\leq~ \frac{1-\theta_k}{\theta_k^\gamma} \bigl( F(x_k) -  F(x) \bigr)
    + L D_h(x, z_k). 
\end{equation}
Finally applying the condition~\eqref{eqn:theta-inequality}
gives the desired result.
\end{proof}

\begin{lemma}\label{lem:theta-explicit}
The sequence $\theta_k = \frac{\gamma}{k+\gamma}$ for $k=0,1,2,\ldots$
satisfies the condition~\eqref{eqn:theta-inequality}.
\end{lemma}

\begin{proof}
With $\theta_k = \frac{\gamma}{k+\gamma}$, we have
\begin{equation}\label{eqn:simple-theta-k+1}
\frac{1-\theta_{k+1}}{\theta_{k+1}^\gamma} 
= \left( 1-\frac{\gamma}{k+1+\gamma}\right) 
  \left(\frac{k+1+\gamma}{\gamma}\right)^\gamma
= \frac{(k+1)(k+1+\gamma)^{\gamma-1}}{\gamma^\gamma}
\end{equation}
and
\begin{equation}\label{eqn:simple-theta-k}
\frac{1}{\theta_k^\gamma} 
=\left(\frac{k+\gamma}{\gamma}\right)^\gamma 
= \frac{(k+\gamma)^\gamma}{\gamma^\gamma}.
\end{equation}
Recall the weighted arithmetic mean and geometric mean inequality
(see, e.g., \cite[Section~2.5]{HLP:52}.), i.e., 
for any positive real numbers $a$, $b$, $\alpha$ and $\beta$, it holds that
\begin{equation}\label{eqn:wagmi2}
a^\alpha b^\beta 
\leq \left(\frac{\alpha a + \beta b}{\alpha+\beta}\right)^{\alpha+\beta}.
\end{equation}
Setting $a=k+1$, $b=k+1+\gamma$, $\alpha=1$ and $\beta=\gamma-1$, we arrive at
\[
(k+1)(k+1+\gamma)^{\gamma-1} \leq 
\left( \frac{k+1+(\gamma-1)(k+1+\gamma)}{1+\gamma-1} \right)^{1+\gamma-1}
=(k+\gamma)^\gamma,
\]
which, 
together with~\eqref{eqn:simple-theta-k+1} and~\eqref{eqn:simple-theta-k},
implies the inequality~(\ref{eqn:theta-inequality}).
\end{proof}

A slightly faster converging sequence $\theta_k$ can be obtained by solving 
the equality in~\eqref{eqn:theta-inequality}.
Since there is no closed-form solution in general, 
we can find $\theta_{k+1}$ as the root of
\begin{equation}\label{eqn:theta-equation}
    \theta^\gamma - \theta_k^\gamma(1-\theta) = 0
\end{equation}
numerically, say, using Newton's method with $\theta_k$ as the starting point. 

\begin{lemma}\label{lem:theta-bound}
Let $\theta_0=1$ and $\theta_{k+1}$ be the solution 
to~\eqref{eqn:theta-equation} for all $k\geq 0$.
Then $\theta_k \leq \frac{\gamma}{k+\gamma}$ for all $k\geq 0$.
\end{lemma}
\begin{proof}
Let $\vartheta_k=\frac{\gamma}{k+\gamma}$ and define another sequence
$\xi_k$ such that $\xi_0=1$ and 
\begin{equation}\label{eqn:theta-mix}
    \frac{1-\xi_{k+1}}{\xi_{k+1}^\gamma} 
    = \frac{1}{\vartheta_k^\gamma},
    \qquad \forall\, k\geq 0.
\end{equation}
Notice that the function 
\[
    \omega(\theta):=\frac{1-\theta}{\theta^{\gamma}}
\]
is monotone decreasing in~$\theta$.
Since $\omega(\vartheta_{k+1})\leq 1/\vartheta_k^\gamma$ 
by Lemma~\ref{lem:theta-explicit}
and $\omega(\xi_{k+1})=1/\vartheta_k^\gamma$ by~\eqref{eqn:theta-mix},
we have $\xi_{k+1}\leq\vartheta_{k+1}$ for all $k\geq 0$.

Next we prove $\theta_k\leq\vartheta_k$ for all $k\geq 0$ by mathematical
induction. This obviously holds for $k=0$ since $\theta_0=\vartheta_0=1$. 
Suppose $\theta_k\leq\vartheta_k$ holds for some $k\geq 0$. 
Then using the facts $\omega(\theta_{k+1})=1/\theta_k^\gamma$ and
$\omega(\xi_{k+1})=1/\vartheta_k^\gamma$, we obtain
$\omega(\theta_{k+1})\geq\omega(\xi_{k+1})$.
Since $\omega$ is monotone decreasing, 
we conclude that $\theta_{k+1}\leq\xi_{k+1}$.
Combining with $\xi_{k+1}\leq\vartheta_{k+1}$ obtained above, we have
$\theta_{k+1}\leq\vartheta_{k+1}$.
This completes the induction.
\end{proof}

\begin{theorem}\label{thm:ABPG-rate}
Suppose Assumption~\ref{asmp:all} holds,
$f$ is $L$-smooth relative to~$h$ on~$C$, 
and $\gamma$ is a uniform TSE of~$D_h$.
If $\theta_k\leq \frac{\gamma}{k+\gamma}$ for all $k\geq 0$, then
the outputs of Algorithm~\ref{alg:ABPG} satisfy,
for any $x\in\dom h$, 
\[
     F(x_{k+1}) -  F(x) \leq 
    \left(\frac{\gamma}{k+\gamma}\right)^\gamma L D_h(x, x_0),
    \qquad \forall\, k\,\geq 0.
\]
\end{theorem}
\begin{proof}
A direct consequence of Lemma~\ref{lem:ABPG-onestep} is, for any $x\in\dom h$,
\[
    \frac{1-\theta_k}{\theta_k^\gamma} \bigl( F(x_k)- F(x)) 
    + L D_h(x, z_k)
    \leq \frac{1-\theta_0}{\theta_0}\bigl( F(x_0)- F(x)\bigr)
    +L D_h(x, z_0).
\] 
Combining with~\eqref{eqn:last-recursion}, we have
\[
    \frac{1}{\theta_k^\gamma} \bigl( F(x_{k+1})- F(x)) 
    + L D_h(x, z_{k+1})
    \leq \frac{1-\theta_0}{\theta_0}\bigl( F(x_0)- F(x)\bigr)
    +L D_h(x, z_0).
\] 
Using $D_h(x,z_{k+1})\geq 0$
and the initializations $\theta_0=1$ and $z_0=x_0$, we obtain
\[
\frac{1}{\theta_k^\gamma} \bigl( F(x_{k+1})- F(x)) \leq L D_h(x, z_0),
\] 
which implies
\[
     F(x_{k+1}) -  F(x) \leq \theta_k^\gamma L D_h(x, x_0).
\]
It remains to apply the condition $\theta_k\leq \frac{\gamma}{k+\gamma}$.
\end{proof}

\subsection{ABPG method with exponent adaptation}
\label{sec:ABPG-expo}

\begin{algorithm}[t]
\caption{ABPG method with exponent adaptation (ABPG-e)}
\label{alg:ABPG-e}
\linespread{1.2}\selectfont
\DontPrintSemicolon
\textbf{input:} initial point $x_0 \in \rint C$, $\gamma_0\geq 2$, 
$\gamma_\mathrm{min}\geq 0$, and $\delta>0$. \;
initialize: $z_0=x_0$, $\gamma_{-1}=\gamma_0$, and $\theta_0=1$. \;
\For{$k= 0,1,2,\dots $}{
$y_k = (1-\theta_k) x_k + \theta_k z_k$ \;
    \Repeat( for $t=0,1,2,\ldots$){
$f(x_{k+1})\leq f(y_k) + \langle\nabla f(y_k), x_{k+1}-y_k\rangle +\theta_k^{\gamma_k}L D_h(z_{k+1},z_k)$
}{
$\gamma_k = \max\{\gamma_{k-1} - \delta t,~\gamma_\mathrm{min}\}$ \;
$z_{k+1} = \argmin_{z \in C} \bigl\{ \ell(z|y_k) 
                + \theta_k^{\gamma_k-1} L D_h(z,z_k)\bigr\}$ \;
$x_{k+1} = (1-\theta_k) x_k + \theta_k z_{k+1}$ \;
} 
\vspace{0.5ex}
choose $\theta_{k+1}\in(0,1]$ such that
$\frac{1-\theta_{k+1}}{\theta_{k+1}^{\gamma_k}}\leq\frac{1}{\theta_k^{\gamma_k}}$\;
\vspace{-0.5ex}
}
\end{algorithm}

The best convergence rate of the ABPG method is obtained with the largest 
uniform TSE for the Bregman distance.
Since it is often hard to determine the largest TSE, we present
in Algorithm~\ref{alg:ABPG-e} a variant of the ABPG method with automatic
exponent adaptation, called the ABPG-e method. 

This method starts with a large $\gamma_0\geq 2$.
During each iteration~$k$, it reduces $\gamma_k$ 
by a small amount $\delta>0$ until some stopping criterion is satisfied.
An obvious choice for the stopping criterion is 
the local triangle-scaling property 
\begin{equation}\label{eqn:local-ts}
D_h(x_{k+1}, y_k)\leq \theta_k^{\gamma_k} D_h(z_{k+1},z_k),
\end{equation}
where $x_{k+1}=(1-\theta_k)x_k + \theta_k z_{k+1}$
and $y_k=(1-\theta_k)x_k + \theta_k z_k$.
According to the proof of Lemma~\ref{lem:ABPG-onestep}, we can 
also use the inequality~\eqref{eqn:ABPG-triangle} as stopping criterion,
which is implied by~\eqref{eqn:local-ts} and the relatively smooth assumption.
For convergence analysis, we only need~\eqref{eqn:ABPG-triangle} to hold, 
which can be less conservative than~\eqref{eqn:local-ts}.
In Algorithm~\ref{alg:ABPG-e}, we use the following inequality as the
stopping criterion
\[
f(x_{k+1})\leq f(y_k) + \langle\nabla f(y_k), x_{k+1}-y_k\rangle +\theta_k^{\gamma_k} L D_h(z_{k+1}, z_k),
\]
which is equivalent to~\eqref{eqn:ABPG-triangle} 
(by subtracting $\Psi(x_{k+1})$ from both sides of the inequality).
In practice, this condition often leads to much faster convergence than 
using~\eqref{eqn:local-ts}.
Computationally, it is slightly more expensive since it needs to evaluate
$f(x_{k+1})$ in addition to $\nabla f(y_k)$ during each inner loop, 
while~\eqref{eqn:local-ts} does not.

The lower bound $\gamma_\mathrm{min}$ can be any known uniform TSE, which
guarantees that the stopping criterion can always be satisfied.
Without such prior information, we can simply set $\gamma_\mathrm{min}=0$.
Since $\gamma_{k+1}\leq\gamma_k$ and $\theta_{k+1}\in (0,1)$, we always have
$\theta_{k+1}^{\gamma_{k+1}} \geq \theta_{k+1}^{\gamma_k}$. Therefore
\[
    \frac{1-\theta_{k+1}}{\theta_{k+1}^{\gamma_{k+1}}}
    \leq \frac{1-\theta_{k+1}}{\theta_{k+1}^{\gamma_k}}
    \leq \frac{1}{\theta_k^{\gamma_k}} .
\]
By replacing inequality~\eqref{eqn:theta-inequality} with the one above and
repeating the analysis in Section~\ref{sec:ABPG-analysis},
we obtain the following result.

\begin{theorem}
Suppose Assumption~\ref{asmp:all} holds,
$f$ is $L$-smooth relative to~$h$ on~$C$, 
and $\gamma_\mathrm{min}$ is a uniform TSE of~$D_h$.
Then the sequences generated by Algorithm~\ref{alg:ABPG-e} satisfy,
for any $x\in\dom h$, 
\[
     F(x_{k+1}) -  F(x) \leq 
    \left(\frac{\gamma_k}{k+\gamma_k}\right)^{\gamma_k} L D_h(x, x_0),
    \qquad \forall\, k\,\geq 0.
\]
\end{theorem}
The convergence rate of ABPG-e is determined by the last value~$\gamma_k$. 
Since we only need to satisfy the local triangle-scaling 
property~\eqref{eqn:local-ts} instead of the uniform 
condition~\eqref{eqn:triangle-scaling}, it is very likely that
$\gamma_k$ is greater than the largest uniform TSE.
However, according to Proposition~\ref{prop:intrinsic-TSE}, when $k\to\infty$,
the limit of $\gamma_k$ (which always exists) cannot be larger than 
the intrinsic TSE $\gammain=2$.
In any case, $\gamma_k$ itself is a \emph{numerical certificate} of an 
empirical convergence rate of $O(k^{-\gamma_k})$, 
albeit one that depends on~$k$.
We always have $\gamma_k \geq \max\{\gamma-\delta,\,\gamma_\mathrm{min}\}$
where $\gamma$ is the uniform TSE.
In our numerical experiments in Section~\ref{sec:experiments}, 
the numerical certificate $\gamma_k$ is mostly close to $\gammain=2$. 

Compared with ABPG, each iteration of ABPG-e may invoke an inner loop that requires additional computation. However, for finishing the same number of~$k$ iterations, the number of extra steps performed by ABPG-e is at most $(\gamma_0-\gamma_{\mathrm{min}})/\delta$. In practice, we always pick $\delta\geq 0.1$, thus the number of extras steps is a constant of at most a few tens, regardless of the number of iterations~$k$.

\section{ABPG methods with gain adaptation}
\label{sec:ABPG-gain}

In this section, we present and analyze an adaptive ABPG method based on 
the concept of intrinsic TSE developed in Section~\ref{sec:intrinsic-TSE}.
Instead of searching for the largest uniform TSE 
as in Algorithm~\ref{alg:ABPG-e}, we can replace line~\ref{lnl:ABPG-zk}
in Algorithm~\ref{alg:ABPG} by
\[
z_{k+1} = \argmin_{z \in C} \left\{ \ell(z|y_k) 
+ G_k \theta_k^{\gamma-1} L D_h(z,z_k)\right\}
\]
and adjust the additional gain $G_k$ while keeping $\gamma=\gammain$ fixed.

Algorithm~\ref{alg:ABPG-g} is such a method with gain adaptation.
During each iteration, the algorithm uses an inner loop to search for an
value of $G_k$ that satisfies
\begin{equation}\label{eqn:func-stop}
f(x_{k+1})\leq f(y_k) + \langle\nabla f(y_k), x_{k+1}-y_k\rangle + G_k\theta_k^{\gamma} L D_h(z_{k+1}, z_k),
\end{equation}
which is true if the following local triangle-scaling property holds:
\begin{equation}\label{eqn:Gk-intrinsic-ts}
D_h(x_{k+1},y_k) 
= D_h\bigl((1-\theta)x_k+\theta z_{k+1}, (1-\theta)x_k+\theta z_k\bigr)
\leq G_k \theta^{\gamma} D_h(z_{k+1},z_k).
\end{equation}
By definition of the intrinsic TSE, such a $G_k$ 
(as a function of $x_{k+1}$, $z_k$ and $z_{k+1}$) 
always exists for $\gamma=\gammain$,
i.e., the stopping criterion for gain adaptation in Algorithm~\ref{alg:ABPG-g} 
can always be satisfied.
In order to obtain fast convergence, we want the value of $G_k$ to be 
as small as possible while still satisfying~\eqref{eqn:func-stop}. 
Therefore, at the beginning of each iteration~$k$, we always try a 
tentative gain that is no larger than $G_{k-1}$: 
$M_k= \max\{ G_{k-1}/\decfac, ~G_\mathrm{min}\}$ with $\rho>1$.
The gain adaptation loop finds the smallest integer $t\geq 0$
such that $G_k=M_k\rho^t$ satisfies the inequality~\eqref{eqn:func-stop}.

Another major difference between Algorithm~\ref{alg:ABPG-g} and the previous
variants of ABPG is that the sequence
$\{\theta_k\}_{k\in\nats}$ in Algorithm~\ref{alg:ABPG-g}
is generated by solving the equation
\begin{equation}\label{eqn:theta-equality-Gk}
\frac{1-\theta_{k+1}}{G_{k+1}\theta_{k+1}^\gamma}
= \frac{1}{G_k\theta_k^\gamma}.
\end{equation}
While an obvious choice is to set $G_\mathrm{min}=1$, we usually set it to 
be much smaller (say $G_\mathrm{min}=10^{-3}$), 
which allows the algorithm to converge much faster.
Since we don't have a priori upper bound on $G_k$, it is hard to 
characterize how fast $\theta_k$ converges to zero.
In fact, $\{\theta_k\}_{k\in\nats}$ may not be a monotone decreasing sequence.
Instead of tracking~$G_k$ and~$\theta_k$ separately, 
we analyze the convergence of the combined quantity $G_k\theta_k^\gamma$.
The following simple lemma will be very useful.

\begin{algorithm}[t]
\caption{ABPG method with gain adaptation (ABPG-g)}
\label{alg:ABPG-g}
\linespread{1.2}\selectfont
\DontPrintSemicolon
\textbf{input:} initial points $x_0 \in C$, $\gamma>1$, 
$\rho>1$ and $G_\mathrm{min}>0$. \;
\textbf{initialize:} $z_0=x_0$, $\theta_0=1$ and $G_{-1}=1$\;
\For{$k= 0,1,2,\dots $}{
    $M_k = \max\{ G_{k-1}/\decfac, ~G_\mathrm{min}\}$ \;
    \Repeat( for $t=0,1,2,\ldots$){
$f(x_{k+1})\leq f(y_k) + \langle\nabla f(y_k), x_{k+1}-y_k\rangle +G_k\theta_k^\gamma L D_h(z_{k+1},z_k)$
    }{
        $G_k = M_k \incfac^{t}$ \;
        \lIf{$k>0$}{compute $\theta_k$ by solving 
        {\small $\displaystyle\frac{1-\theta_k}{G_k\theta_k^\gamma}=\frac{1}{G_{k-1}\theta_{k-1}^\gamma}$}}
        \vspace{-1ex}
        $y_k = (1-\theta_k) x_k + \theta_k z_k$ \;
        $z_{k+1} = \argmin_{z \in C} \left\{ \ell(z|y_k) 
        + G_k \theta_k^{\gamma-1} L D_h(z,z_k)\right\}$ \;
        $x_{k+1} = (1-\theta_k) x_k + \theta_k z_{k+1}$ \;
    }
}
\end{algorithm}

\begin{lemma}\label{lem:power-diff}
For any $\alpha,\beta>0$ and $\gamma\geq1$, the following inequality holds:
\[
\alpha^\gamma - \beta^\gamma ~\leq~ \gamma(\alpha-\beta)\alpha^{\gamma-1}.
\]
\end{lemma}
\begin{proof}
The case of $\gamma=1$ is obvious. Assume $\gamma>1$.
The desired inequality is equivalent to
\[
\alpha^{\gamma-1}\beta 
~\leq~ \frac{(\gamma-1)\alpha^\gamma + \beta^\gamma}{\gamma}
~=~ \frac{(\gamma-1)\alpha^\gamma + 1\cdot \beta^\gamma}{(\gamma-1) + 1}.
\]
Applying the weighted arithmetic and geometric mean 
inequality~\eqref{eqn:wagmi2}, we have
\[
    \frac{(\gamma-1)\alpha^\gamma + 1\cdot\beta^\gamma}{(\gamma-1)+1}
~\geq~\Bigl((\alpha^\gamma)^{\gamma-1}(\beta^\gamma)^{1}\Bigr)^\frac{1}{\gamma}
~=~ \alpha^{\gamma-1}\beta ,
\]
which completes the proof.
\end{proof}

\begin{theorem}\label{thm:adapt-G}
Suppose Assumption~\ref{asmp:all} holds,
$f$ is $L$-smooth relative to~$h$ on~$C$, 
and $\gamma=\gammain$ is the intrinsic TSE of~$D_h$.
Then the sequences generated by Algorithm~\ref{alg:ABPG-g} satisfy,
for any $x\in\dom h$, 
\begin{equation}\label{eqn:ABPG-g-rate}
 F(x_{k+1})- F(x) ~\leq~ \left(\frac{\gamma}{k+\gamma}\right)^\gamma
\overline{G}_k L D_h(x, x_0),
\qquad \forall\, k\geq 0,
\end{equation}
where $\overline{G}_k$ is a weighted geometric mean of the gains at each step:
\begin{equation}\label{eqn:wgm-Gk}
\overline{G}_k~=~\left(G_0^\gamma G_1\cdots G_k\right)^{\frac{1}{k+\gamma}}.
\end{equation}
\end{theorem}
\begin{proof}
We follow the same steps as in Section~\ref{sec:ABPG-analysis}. 
In light of~\eqref{eqn:Gk-intrinsic-ts}, 
the inequality~\eqref{eqn:ABPG-triangle} becomes
\begin{equation}\label{eqn:ABPG-m-rs-ts}
 F(x_{k+1}) \leq \ell(x_{k+1}|y_k) + G_k\theta_k^\gamma L D_h(z_{k+1},z_k),
\end{equation}
and the inequality~\eqref{eqn:last-recursion} becomes
\begin{equation}\label{eqn:converge-pre-G}
\frac{1}{G_k\theta_k^\gamma}\bigl( F(x_{k+1})- F(x)\bigr)+L D_h(x,z_{k+1}) 
~\leq~ \frac{1-\theta_k}{G_k\theta_k^\gamma} \bigl( F(x_k) -  F(x) \bigr)
    + L D_h(x, z_k). 
\end{equation}
Plugging in the equality~\eqref{eqn:theta-equality-Gk}, we obtain
\begin{equation}\label{eqn:converge-recursion-G}
\frac{1-\theta_{k+1}}{G_{k+1}\theta_{k+1}^\gamma} \bigl( F(x_{k+1})- F(x)\bigr)+L D_h(x, z_{k+1}) 
~\leq~ \frac{1-\theta_k}{G_k\theta_k^\gamma} \bigl( F(x_k) -  F(x) \bigr)
    + L D_h(x, z_k). 
\end{equation}
Then the same arguments in the proof of Theorem~\ref{thm:ABPG-rate} lead to
\begin{equation}\label{eqn:Gk-thetak-bound}
 F(x_{k+1})- F(x) ~\leq~ G_k\theta_k^\gamma L D_h(x, x_0).
\end{equation}

Next we derive an upper bound for $G_k\theta_k^\gamma$. 
For convenience, let's define for $k=0,1,2,\ldots$,
\[
A_k = \frac{1}{G_k\theta_k^\gamma}, \qquad
a_{k+1} = \frac{1}{G_{k+1}\theta_{k+1}^{\gamma-1}}. 
\]
Then~\eqref{eqn:theta-equality-Gk} implies $a_{k+1} = A_{k+1}-A_k$. 
Moreover, we have
\begin{equation}\label{eqn:Ak1-Ak}
A_{k+1} ~=~ \frac{1}{G_{k+1}\theta_{k+1}^\gamma}
~=~ G_{k+1}^{\frac{1}{\gamma-1}} a_{k+1}^{\frac{\gamma}{\gamma-1}}
~=~ 
G_{k+1}^{\frac{1}{\gamma-1}} \Bigl(A_{k+1}-A_k\Bigr)^{\frac{\gamma}{\gamma-1}}.
\end{equation}
Applying Lemma~\ref{lem:power-diff} with $\alpha=A_{k+1}^{1/\gamma}$
and $\beta=A_k^{1/\gamma}$, we obtain
\[
A_{k+1} - A_k
~=~ \Bigl(A_{k+1}^\frac{1}{\gamma}\Bigr)^\gamma
    - \Bigl(A_k^\frac{1}{\gamma}\Bigr)^\gamma 
~\leq~ \gamma\Bigl( A_{k+1}^\frac{1}{\gamma} - A_k^\frac{1}{\gamma}\Bigr)
    A_{k+1}^{\frac{\gamma-1}{\gamma}} .
\]
Combining with~\eqref{eqn:Ak1-Ak} yields
\[
A_{k+1} 
~=~G_{k+1}^{\frac{1}{\gamma-1}}\Bigl(A_{k+1}-A_k\Bigr)^{\frac{\gamma}{\gamma-1}}
~\leq~G_{k+1}^{\frac{1}{\gamma-1}} \gamma^{\frac{\gamma}{\gamma-1}}
\Bigl( A_{k+1}^\frac{1}{\gamma} - A_k^\frac{1}{\gamma}\Bigr)^{\frac{\gamma}{\gamma-1}}
A_{k+1}
\]
We can eliminate the common factor $A_{k+1}$ on both sides of the above 
inequality to obtain
\[
1 ~\leq~ G_{k+1}^{\frac{1}{\gamma-1}} \gamma^{\frac{\gamma}{\gamma-1}}
\Bigl( A_{k+1}^\frac{1}{\gamma} - A_k^\frac{1}{\gamma}\Bigr)^{\frac{\gamma}{\gamma-1}},
\]
which implies
\[
A_{k+1}^\frac{1}{\gamma} - A_k^\frac{1}{\gamma}
~\geq~ \frac{1}{\gamma\, G_{k+1}^{1/\gamma}},
\qquad k=0,1,2,\ldots.
\]
Summing the above inequality from step~$0$ to~$k-1$ and using $A_0=1/G_0$, 
we have
\[
A_k^{\frac{1}{\gamma}} 
~\geq~ \sum_{t=1}^k\frac{1}{\gamma\,G_t^{1/\gamma}}+A_0^{\frac{1}{\gamma}}
~=~ \sum_{t=1}^k\frac{1}{\gamma\, G_t^{1/\gamma}} + \frac{1}{G_0^{1/\gamma}}
~=~ \frac{1}{\gamma}\left(\sum_{t=1}^k\frac{1}{G_t^{1/\gamma}} 
    + \frac{\gamma}{G_0^{1/\gamma}}\right).
\]
Using the weighted arithmetic and geometric mean inequality
(e.g., \cite[Section~2.5]{HLP:52}) gives
\[
\sum_{t=1}^k\frac{1}{G_t^{1/\gamma}} + \frac{\gamma}{G_0^{1/\gamma}}
~\geq~ (k+\gamma)\left(\biggl(\frac{1}{G_0^{1/\gamma}}\biggr)^\gamma 
    \frac{1}{G_1^{1/\gamma}} \cdots \frac{1}{G_k^{1/\gamma}}
    \right)^{\frac{1}{k+\gamma}}
=~(k+\gamma)\Bigl(G_0^\gamma G_1\cdots G_k\Bigr)^{\frac{-1}{\gamma(k+\gamma)}}.
\]
Combining the last two inequalities above, we arrive at
\[
A_k ~\geq~ \left(\frac{k+\gamma}{\gamma}\right)^\gamma
\left(G_0^\gamma G_1\cdots G_k\right)^{\frac{-1}{k+\gamma}}.
\]
Therefore, 
\[
G_k\theta_k^\gamma ~=~ \frac{1}{A_k}
~\leq~ \left(\frac{\gamma}{k+\gamma}\right)^\gamma
\left(G_0^\gamma G_1\cdots G_k\right)^{\frac{1}{k+\gamma}}.
\]
Finally, substituting the inequality above
into~\eqref{eqn:Gk-thetak-bound} gives the desired result.
\end{proof}

We note that the geometric mean $\overline{G}_k$ in~\eqref{eqn:wgm-Gk} can be 
much smaller than the average (arithmetic mean) of $\{G_0,G_1,\ldots,G_k\}$.
Under the assumption of uniform Lipschitz smoothness~\eqref{eqn:uniform-smooth},
Nesterov \cite{Nesterov13composite} proposed an accelerated gradient method
with non-monotone line search. However, the complexity obtained there still
depends on the global Lipschitz constant $L$, more specifically, replacing
$\overline{G}_k L$ in~\eqref{eqn:ABPG-g-rate} with $\rho L$ when $\gamma=2$.
Our result in~\eqref{eqn:ABPG-g-rate} can be tighter if the local
Lipschitz constants are smaller than $L$ (equivalently with $G_k<1$).

\paragraph{Total number of oracle calls.}\label{remark:ABPG-e-oracles}
In order to estimate the overhead of the gain-adaptation procedure,
we follow the approach of \cite[Lemma~4]{Nesterov13composite}.
Notice that each inner loop needs to call a gradient oracle to compute
$\nabla f(y_k)$, and also $f(x_{k+1})$ when we use~\eqref{eqn:func-stop}
as the stopping criterion for gain adaptation.
Let $n_i\geq 1$ be the number of calls of the oracle (for $\nabla f(y_k)$) at
the $i$th iteration, for $i=0, \ldots, k$. Then
\[
G_{i+1} = \max\{G_i/\rho, G_\mathrm{min}\} \rho^{n_i-1} \geq G_i \rho^{n_i-2}, 
\qquad i=0,\ldots, k-1.
\]
Thus
\[
n_i \leq 2 + \log_{\rho}\frac{G_{i+1}}{G_i} 
= 2 + \frac{1}{\ln\rho}\ln\frac{G_{i+1}}{G_i} .
\]
Therefore, the total number of oracle calls is
\[
N_k = \sum_{i=0}^k n_i 
\leq \sum_{i=0}^k \left(2 + \frac{1}{\ln\rho}\ln\frac{G_{i+1}}{G_i} \right)
= 2(k+1) + \frac{1}{\ln\rho}\ln\frac{G_k}{G_0} .
\]
Roughly speaking, on average each iteration need two oracle calls
(unless $G_k$ becomes very large).

\paragraph{An explicit update rule for $\theta_k$.}
As an alternative to calculating $\theta_{k+1}$ by solving the 
equation~\eqref{eqn:theta-equality-Gk}, we can also use the following
explicit update rule:
\[
\frac{1}{\theta_{k+1}} 
~=~ \frac{\gamma\alpha_k}{1+\alpha_k(\gamma-1)}\frac{1}{\theta_k}
+ \frac{1}{1+\alpha_k(\gamma-1)}, 
\]
where $\alpha_k=G_{k+1}/G_k$ for $k=0,1,2,\dots$.
This recursion is obtained by solving a linearized equation 
of~\eqref{eqn:theta-equality-Gk}.
In particular, if $\alpha_k=1$ for all $k\geq 0$, then this formula
produces $\theta_k=\gamma/(k+\gamma)$.
The sequence $\{\theta_k\}_{k\in\nats}$ generated this way
satisfies an inequality obtained by replacing the ``$=$'' sign with ``$\leq$''
in~\eqref{eqn:theta-equality-Gk}. 
Although Theorem~\ref{thm:adapt-G} does not apply to this sequence,
it often has comparable or even faster performance in practice, 
especially when the $\alpha_k$'s are close to~1.

\subsection{Towards the $O(k^{-2})$ convergence rate}
\label{sec:towards-k-2}

Proposition~\ref{prop:intrinsic-TSE} shows that the intrinsic TSE $\gammain=2$ 
for all Bregman distances $D_h$ where~$h$ is convex and twice continuously
differentiable. This covers most Bregman distances of practical interest.
If we run the ABPG-g method (Algorithm~\ref{alg:ABPG-g}) with $\gamma=2$, 
then Theorem~\ref{thm:adapt-G} states that the convergence rate is 
$O(\overline{G}_k k^{-2})$.
In order to obtain the $O(k^{-2})$ convergence rate, we need 
$\overline{G}_k$ to be $O(1)$.
In this subsection, we discuss the asymptotic behavior of $G_k$, which dominate the geometric mean $\overline{G}_k$ when $k\to \infty$. 

We focus on the concrete case of KL divergence, which is a widely used in first-order optimization algorithms.
For the KL divergence, $h(x)=\sum_{i=1}^n x^{(i)}\log x^{(i)}$ and 
$\nabla^2 h(x) = \mathrm{diag}\bigl(\frac{1}{x^{(1)}}, \ldots, \frac{1}{x^{(n)}}\bigr)$, thus according to the derivations in Section~\ref{sec:bound-ts-gain},
\begin{equation}\label{eqn:gain-KL-bd}
G_{\theta_k}(x_k, z_k, z_{k+1}) 
\leq \eigmax\left(\nabla^2 h(v_k)^{-1/2} \nabla^2 h(u_k) \nabla^2 h(v_k)^{-1/2}\right)
= \max_{i\in\{1,\ldots,n\}} \frac{v_k^{(i)}}{u_k^{(i)}},
\end{equation}
where $u_k\in\bigl[(1-\theta_k)x_k+\theta_k z_k, (1-\theta_k)x_k+\theta_k z_{k+1}\bigr]$ and $v_k\in\bigl[z_k, z_{k+1}\bigr]$.
Suppose the sequence $\{x_k\}$ converges to the optimal solution $x_\star$ and
$\theta_k\to 0$, then have $u_k\to x_\star$.
If $x_\star$ is an interior point of the positive orthant or the simplex, meaning $x_\star^{(i)}>0$ for all coordinates~$i$, 
Then we see from~\eqref{eqn:gain-KL-bd} that the bound on 
$G_{\theta_k}(x_k, z_k, z_{k+1})$
depends on how close $x_\star$ is close to the boundary
(assuming $v_k$ is bounded). 

The most interesting case is when the optimal solution $x_\star$ is on the boundary, i.e., when $x_\star^{(i)}=0$ for some coordinates $i$. 
In fact, in our numerical examples on the D-optimal design problem and Poisson linear inverse problem, most of the solutions are on the boundary.
However, we emphasize that the iterates generated by the ABPG algorithm is never on the boundary, but may only converge to the boundary; see Assumption~A, especially A.5.
Our analysis applies to this case as well. 
In particular, we can show that if both sequences $\{x_k\}$ and $\{z_k\}$ converge, then $x_k^{(i)}\to 0$ implies $z_k^{(i)}\to 0$ and the convergence rate of 
$x_k^{(i)}$ is no faster than that of $z_k^{(i)}$.
(Here $x_k^{(i)}\to 0$ means $\lim_{k\to\infty}x_k^{(i)}=0$.)
More precisely, we have the following lemma.

\begin{lemma} \label{lem:xk-zk-converge}
Suppose an algorithm generates two sequences $\{x_k\}$ and $\{z_k\}$ in the strictly positive orthant, satisfying $x_0=z_0$ and $x_{k+1}=(1-\theta_k)x_k + \theta_k z_{k+1}$ for all $k\geq 0$. Then
\begin{itemize}
\item[(a)] 
If $\{x_k\}$ converges and $x_k^{(i)}\to 0$ for some coordinate~$i$,
then there must exists an subsequence of $\{z_{k}^{(i)}\}$ that converges to $0$. 
\item[(b)] Suppose both sequences $\{x_k\}$ and $\{z_k\}$ converge.
If $x_k^{(i)}\to 0$ for some~$i$, then it converges at a rate that is no faster than $z_k^{(i)}$ in the following sense:
For any monotone decreasing sequence $\{r_k\}$ that converges to~$0$ and satisfies $z_k^{(i)}\geq r_k$ for all $k\geq 0$, we have $x_k^{(i)}\geq r_k$ for all $k\geq 0$.
In particular, we can choose $r_k$ to be the monotone lower envelop of $z_k^{(i)}$, i.e., $r_k=\min\{z_0^{(i)}, z_1^{(i)},\ldots, z_k^{(i)}\}$.
\end{itemize}
\end{lemma}

\begin{proof}

By the update rule $x_{k+1}=(1-\theta_k) + \theta_k z_{k+1}$, we know that each $x_k$ is a convex combination of the points $\{z_0=x_0, z_1, \ldots, z_k\}$, which all lie in the strictly positive orthant. Therefore,
\[
x_k^{(i)}\geq \min\{z_0^{(i)},z_1^{(i)},\ldots,z_k^{(i)}\} ~>~ 0, \qquad\forall i, k.
\vspace{-1ex}
\]

\emph{Part~(a).}
Suppose $x_k^{(i)}\to 0$ but there is no subsequence of $\{z_k^{(i)}\}$ converging to~$0$. 
Then there must exist an~$\epsilon>0$ such that $z_k^{(i)}>\epsilon$ for all $k\geq 0$.
Since $x_k$ is a convex combination of $\{z_0, z_1, \ldots, z_k\}$, this implies 
$$x_k^{(i)}\geq \min\{z_0^{(i)},z_1^{(i)},\ldots,z_k^{(i)}\}>\epsilon, \qquad \forall\, k\geq 0,$$
which contradicts with the assumption that $x_k^{(i)}\to 0$.
Therefore, there must exists an subsequence of $\{z_{k}^{(i)}\}$ that converges to $0$. 

\emph{Part~(b).}
Suppose $\{r_k\}$ is monotone decreasing and converges to~$0$.
If $z_k^{(i)}\geq r_k$ for all $k\geq 0$, then
\[
x_k^{(i)}\geq \min\{z_0^{(i)},z_1^{(i)},\ldots,z_k^{(i)}\}
\geq \min\{r_0,r_1,\ldots,r_k\} = r_k, \qquad \forall\,k\geq 0.
\]
In this sense, $x_k^{(i)}\to 0$ at a rate that is no faster than $z_k^{(i)}$.
\end{proof}

\smallskip

According to the bound in~\eqref{eqn:gain-KL-bd} and Lemma~\ref{lem:xk-zk-converge}, if both sequences $\{x_k\}$ and $\{z_k\}$ converge, then
\[
G_{\theta_k}(x_k, z_k, z_{k+1}) 
\leq \max_{i\in\{1,\ldots,n\}} \frac{v_k^{(i)}}{u_k^{(i)}}
\approx \max_{i\in\{1,\ldots,n\}} \frac{z_k^{(i)}}{x_k^{(i)}},
\] 
which can be bounded by a constant, since $x_k^{(i)}\to 0$ at a rate that is no faster than $z_k^{(i)}$.
In this case, we have $G_k$ in Algorithm~\ref{alg:ABPG-g} bounded by a constant asymptotically, thus the convergence rate is $O(k^2)$. 

For the IS divergence, $h(x)= \sum_{i=1}^n -\log x^{(i)}$ and 
$\nabla^2 h(x) = \mathrm{diag}\bigl((\frac{1}{x^{(1)}})^2, \ldots, (\frac{1}{x^{(n)}})^2\bigr)$, thus the situation is very similar to the KL divergence.

However, we are not able to prove the convergence of the sequences 
$\{x_k\}$ and $\{z_k\}$ without additional assumptions 
(such as relative strong convexity). 
Indeed, to the best of our knowledge, convergence of these sequences have not been established even under the classical uniform Lipschitz condition.
Therefore, an a priori theoretical guarantee of the $O(k^{-2})$ rate
seems to be out of reach in general, which seems to coroborate the recent 
result in \cite{Dragomir2019} that the $O(k^{-1})$ rate cannot be improved in 
in general for the class of relatively smooth functions.

Nevertheless, we would like to reiterate the remarks at the end of 
Sections~\ref{sec:relative-smooth}. In particular, the class of relatively 
smooth functions is very large, and the lower bound in \cite{Dragomir2019}
is established with a worst-case function with pathological nonsmooth behavior.
In practical applications, we always work with one particular reference function which may possess structural properties that allow fast convergence.
In Algorithm~\ref{alg:ABPG-g}, the sequence $\{G_k\}$ is readily available 
as part of the computation and we can easily check the magnitude of 
$\overline{G}_k$.
Whenever it is small, we obtain a numerical certificate that the algorithm
did converge with the $O(k^{-2})$ rate.
This is exactly what we observe in the numerical experiments in
Section~\ref{sec:experiments}.

\section{Accelerated Bregman dual averaging method}
\label{sec:ABDA}

In this section, we present an accelerated Bregman dual averaging (ABDA) method
under the relative smoothness assumption.
This method extends Nesterov's accelerated dual averaging method
(\cite{Nesterov05smooth} and \cite[Algorithm~3]{Tseng08}) 
to the relatively smooth setting.
Here we focus on a simple variant in Algorithm~\ref{alg:ABDA} based on the 
uniform triangle-scaling property, although it is also possible to develop 
more sophisticated variants with automatic exponent or gain adaptation.

In Algorithm~\ref{alg:ABDA}, Line~\ref{lnl:psi-k-recursion} defines a sequence 
of functions $\{\psi_k\}_{k\in\nats}$ starting with $\psi_0\equiv 0$:
\begin{equation}\label{eqn:psi-k-recursion}
\psi_{k+1}(x) := \psi_k(x) + \theta_k^{1-\gamma} \ell(x|y_k).
\end{equation}
In other words, $\psi_{k+1}$ is a weighted sum of the lower approximations 
in~\eqref{eqn:sanwitch} constructed at $y_0,\ldots,y_k$:
\begin{equation}\label{eqn:psi-k-sum}
    \psi_{k+1}(x) = \sum_{t=0}^k \theta_t^{1-\gamma} \ell(x|y_t) .
\end{equation}
Line~\ref{lnl:ABDA-zk} in Algorithm~\ref{alg:ABDA} can be written as
\begin{equation}\label{eqn:ABDA-zk-simplified}
z_{k+1} = \argmin_{z\in C}\, \bigl\{
    \left\langle g_k, z\right\rangle + \vartheta_k \Psi(z) + L h(z) \bigr\}
\end{equation}
where 
\[
    g_k = \sum_{t=1}^k\theta_t^{1-\gamma}\nabla f(y_t), \qquad
    \vartheta_k = \sum_{t=1}^k\theta_t^{1-\gamma}. 
\]
When implementing Algorithm~\ref{alg:ABDA}, we only need to keep track 
of $g_k$ and $\vartheta_k$, and there is no need to maintain the abstract
form of $\psi_k(x)$.
Here our assumption of $C$ and $\Psi$ being simple means that the minimization
problem in~\eqref{eqn:ABDA-zk-simplified} can be solved efficiently.
This requirement is equivalent to that for the BPG method~\eqref{eqn:BPG}
and all variants of the ABPG methods in this paper.

\begin{algorithm}[t]
\caption{Accelerated Bregman dual averaging (ABDA) method}
\label{alg:ABDA}
\linespread{1.2}\selectfont
\DontPrintSemicolon
\textbf{input:} initial point $z_0\in\rint C$ and $\gamma>1$. \;
initialize: $x_0=z_0$, $\psi_0(x)\equiv 0$, and $\theta_0=1$. \;
\For{$k=0,1,2,\dots $}{
\nl $y_k := (1-\theta_k) x_k + \theta_k z_k$  \;
\nl\label{lnl:psi-k-recursion}%
$\psi_{k+1}(x) := \psi_k(x) + \theta_k^{1-\gamma} \ell(x|y_k)$ \;
\nl\label{lnl:ABDA-zk}%
$z_{k+1} := \argmin_{z \in C} \bigl\{ \psi_{k+1}(z) + L h(z)\bigr\}$ \;
\nl $x_{k+1} := (1-\theta_k) x_k + \theta_k z_{k+1}$ \;
\nl\label{lnl:ABDA-theta}%
find $\theta_{k+1}\in(0,1]$ such that
$\frac{1-\theta_{k+1}}{\theta_{k+1}^\gamma} = \frac{1}{\theta_k^\gamma}$\;
}
\end{algorithm}

Algorithm~\ref{alg:ABDA} (line~\ref{lnl:ABDA-theta}) requires the sequence 
$\{\theta_k\}_{k\in\nats}$ satisfy
\begin{equation}\label{eqn:theta-equality}
    \frac{1-\theta_{k+1}}{\theta_{k+1}^\gamma} = \frac{1}{\theta_k^\gamma},
    \qquad \forall\, k\geq 0.
\end{equation}
Under this condition, we can show 
\begin{equation}\label{eqn:theta-sum}
\vartheta_k = 
\sum_{i=0}^k \theta_i^{1-\gamma} = \frac{1}{\theta_k^\gamma}.
\end{equation} 
To see this, we use induction. Clearly it holds for $k=0$ 
if we choose $\theta_0=1$.
Suppose it holds for some $k\geq 0$,
then in light of~\eqref{eqn:theta-sum} and~\eqref{eqn:theta-equality},
\[
\vartheta_{k+1} = \sum_{i=0}^{k+1}\frac{1}{\theta_i^{\gamma-1}} 
= \frac{1}{\theta_k^\gamma} + \frac{1}{\theta_{k+1}^{\gamma-1}}
= \frac{1-\theta_{k+1}}{\theta_{k+1}^\gamma}+\frac{1}{\theta_{k+1}^{\gamma-1}}
= \frac{1-\theta_{k+1}+\theta_{k+1}}{\theta_{k+1}^\gamma}
= \frac{1}{\theta_{k+1}^\gamma}.
\]
Therefore the inequality~\eqref{eqn:theta-sum} holds for all $k\geq 0$.

To analyze the convergence of Algorithm~\ref{alg:ABDA}, we need the following
simple variant of Lemma~\ref{lem:basic-lemma-D}.
\begin{lemma}\label{lem:basic-lemma-h}
Suppose~$h$ is convex and differentiable on $\rint C$. 
For any closed convex function $\varphi$, if
\[
    z = \argmin_{x\in C} \, \bigl\{ \varphi(x) + h(x) \bigr\}
\]
and $h$ is differentiable at~$z$, then
\[
    \varphi(x) + h(x) \geq \varphi(z) + h(z) + D_h(x,z), 
    \quad \forall\, x\in \dom h.
\]
\end{lemma}

\begin{lemma}\label{lem:ABDA-onestep}
Suppose Assumption~\ref{asmp:all} holds, 
$f$ is $L$-smooth relative to~$h$ on~$C$,
and $\gamma$ is a uniform TSE of~$D_h$.
Then the sequences generated by Algorithm~\ref{alg:ABDA} satisfy,
for all $x\in\dom h$ and all $k\geq1$, 
\begin{equation}\label{eqn:ABDA-onestep}
\frac{1-\theta_{k+1}}{\theta_{k+1}^\gamma} F(x_{k+1})
-\psi_{k+1}(z_{k+1}) - L h(z_{k+1}) 
~\leq~ \frac{1-\theta_k}{\theta_k^\gamma}  F(x_k) 
- \psi_k(z_k) - L h(z_k).
\end{equation}
\end{lemma}

\begin{proof}
We can start with the inequality~\eqref{eqn:dual-avg-start}:
\begin{eqnarray}
 F(x_{k+1})
&\leq& (1-\theta_k)\ell(x_k|y_k) + \theta_k \ell(z_{k+1}|y_k) 
    + \theta_k^\gamma L D_h(z_{k+1}, z_k) \nonumber \\
&=& (1-\theta_k)\ell(x_k|y_k) + \theta_k^{\gamma} \left( \theta_k^{1-\gamma} 
    \ell(z_{k+1}|y_k) + L D_h(z_{k+1}, z_k) \right) \nonumber \\
&\leq& (1-\theta_k) F(x_k) + \theta_k^{\gamma} \left( \theta_k^{1-\gamma} 
    \ell(z_{k+1}|y_k) + L D_h(z_{k+1}, z_k) \right) .
\label{eqn:dual-avg-dzz}
\end{eqnarray}
Notice that for $k\geq 1$, $z_k$ is the minimizer of $\psi_k(z)+L h(z)$ 
over~$C$. We use Lemma~\ref{lem:basic-lemma-h} to obtain
\[
\psi_k(z_k) + L h(z_k) + L D_h(z_{k+1}, z_k) \leq \psi_k(z_{k+1})+L h(z_{k+1}),
\]
which gives 
\begin{equation}\label{eqn:bound-div-zz}
L D_h(z_{k+1}, z_k) \leq \psi_k(z_{k+1}) + L h(z_{k+1})-\psi_k(z_k)-L h(z_k).
\end{equation}
Combining the inequalities~\eqref{eqn:dual-avg-dzz} 
and~\eqref{eqn:bound-div-zz}, we obtain
\begin{eqnarray*}
 F(x_{k+1})
&\leq& (1-\theta_k) F(x_k) + \theta_k^{\gamma} \left( 
    \theta_k^{1-\gamma} \ell(z_{k+1}|y_k) 
    + \psi_k(z_{k+1}) + L h(z_{k+1}) - \psi_k(z_k) - L h(z_k) \right) \\
&=& (1-\theta_k) F(x_k) + \theta_k^{\gamma} \bigl( 
    \psi_{k+1}(z_{k+1}) + L h(z_{k+1}) - \psi_k(z_k) - L h(z_k) \bigr),
\end{eqnarray*}
where in the last equality we used recursive definition of $\psi_{k+1}$
in~\eqref{eqn:psi-k-recursion}.
Dividing both sides of the above inequality by $\theta_k^\gamma$, we have
\[
\frac{1}{\theta_k^\gamma} F(x_{k+1})
~\leq~ \frac{1-\theta_k}{\theta_k^\gamma}  F(x_k) 
+\psi_{k+1}(z_{k+1}) + L h(z_{k+1}) - \psi_k(z_k) - L h(z_k).
\]
Using~\eqref{eqn:theta-equality} and rearranging terms gives the desired 
result~\eqref{eqn:ABDA-onestep}, which holds for $k\geq 1$.
\end{proof}

\begin{theorem}\label{thm:ABDA-rate}
Suppose Assumption~\ref{asmp:all} holds, 
$f$ is $L$-smooth relative to~$h$ on~$C$,
and $\gamma$ is a uniform TSE of~$D_h$.
The sequences generated by Algorithm~\ref{alg:ABDA} satisfy:
\begin{itemize}
\item[(a)] If $z_0=\argmin_{z\in C} h(z)$, then for any $x\in\dom h$, 
\begin{equation}\label{eqn:ABDA-rate-z0-min}
 F(x_{k+1}) -  F(x) \leq \left(\frac{\gamma}{k+\gamma}\right)^\gamma 
L \bigl(h(x) - h(z_0)\bigr),
\qquad \forall\, k\,\geq 0;
\end{equation}
\item[(b)] Otherwise, for any $x\in\dom h$,
\begin{equation}\label{eqn:ABDA-rate-no-min}
 F(x_{k+1}) -  F(x) \leq \left(\frac{\gamma}{k+\gamma}\right)^\gamma 
L \bigl(h(x) - h(z_1) + D_h(z_1,z_0)\bigr),
\qquad \forall\, k\,\geq 0.
\end{equation}
\end{itemize}
\end{theorem}

\begin{proof}
If $z_0=\argmin_{z\in C} h(z)$, we use the definition $\psi_0\equiv 0$ 
to conclude that
\[
    z_0 = \argmin_{z\in C} \,\bigl\{\psi_0(z) + L h(z) \bigr\}.
\]
In this case, we can extend the result of Lemma~\ref{lem:ABDA-onestep} to 
hold for all $k\geq 0$.
Applying the inequality~\eqref{eqn:ABDA-onestep} for iterations $0,1,\ldots,k$,
we obtain
\[
\frac{1-\theta_{k+1}}{\theta_{k+1}^\gamma}  F(x_{k+1})
-\psi_{k+1}(z_{k+1}) - L h(z_{k+1}) 
~\leq~ \frac{1-\theta_0}{\theta_0^\gamma}  F(x_0) - \psi_0(z_0) - L h(z_0)
~=~ - L h(z_0),
\]
where we used $\theta_0=1$ and $\psi_0\equiv 0$.
Next using~\eqref{eqn:theta-equality} and rearranging terms, we have
\begin{eqnarray}
\frac{1}{\theta_k^\gamma}  F(x_{k+1})
&\leq& \psi_{k+1}(z_{k+1}) + L h(z_{k+1}) - L h(z_0) \nonumber \\
&\leq& \psi_{k+1}(x) + L h(x) - L h(x_0)
\label{eqn:ABDA-pre-k}\\
&=& \sum_{t=0}^k \theta_t^{1-\gamma}\ell(x|y_t) + L \bigl(h(x)-h(z_0)\bigr) 
\nonumber \\
&\leq & \sum_{t=0}^k \theta_t^{1-\gamma} F(x) + L \bigl(h(x)-h(z_0)\bigr) 
\nonumber \\
&=& \frac{1}{\theta_k^\gamma} F(x) + L \bigl(h(x)-h(z_0)\bigr) ,
\label{eqn:ABDA-k-steps}
\end{eqnarray}
where the second inequality used the fact that $z_{k+1}$ is the minimizer of
$\psi_{k+1}(z)+Lh(z)$, the third inequality used $\ell(x|y_t)\leq F(x)$,
and the last equality used~\eqref{eqn:theta-sum}.
Rearranging terms of~\eqref{eqn:ABDA-k-steps} yields
\[
 F(x_{k+1}) -  F(x) \leq \theta_k^\gamma L\bigl(h(x)-h(z_0)\bigr).
\]
According to Lemma~\ref{lem:theta-bound}, we have 
$\theta_k\leq\frac{\gamma}{k+\gamma}$ if~\eqref{eqn:theta-equality} holds,
which gives ~\eqref{eqn:ABDA-rate-z0-min}.

If $z_0\neq\argmin_{z\in C} h(z)$, then we can only
apply~\eqref{eqn:ABDA-onestep} for $k\geq 1$ to obtain
\begin{eqnarray*}
\frac{1}{\theta_k^\gamma}  F(x_{k+1}) - \psi_{k+1}(z_{k+1}) - L h(z_{k+1})
&\leq& \frac{1-\theta_1}{\theta_1^\gamma} F(x_1) - \psi_1(z_1)- L h(z_1) \\
&=& \frac{1}{\theta_0^\gamma} F(x_1) - \theta_0^{1-\gamma}\ell(z_1|y_0)
    - L h(z_1) \\
&=&  F(z_1) - \ell(z_1|z_0)- L h(z_1) \\
&\leq& L D_h(z_1, z_0) - L h(z_1),
\end{eqnarray*}
where the first equality used~\eqref{eqn:theta-equality}, the second 
equality used $\theta_0=1$, $y_0=z_0$ and $x_1=z_1$,
and the last inequality is due to relative smoothness:
$ F(z_1)\leq \ell(z_1|z_0) + L D_h(z_1, z_0)$.
Therefore, 
\begin{eqnarray*}
\frac{1}{\theta_k^\gamma}  F(x_{k+1})
&\leq& \psi_{k+1}(z_{k+1}) + L h(z_{k+1}) + L D_h(z_1, z_0) - L h(z_1) \\
&\leq& \frac{1}{\theta_k^\gamma} F(x) 
+ L \bigl(h(x) - h(z_1) + D_h(z_1, z_0) \bigr), 
\end{eqnarray*}
where the last inequality repeats the arguments from~\eqref{eqn:ABDA-pre-k}
to~\eqref{eqn:ABDA-k-steps}.
Rearranging terms leads to
\[
 F(x_{k+1}) -  F(x)
~\leq~ \theta_k^\gamma L \bigl(h(x) - h(z_1) + D_h(z_1, z_0) \bigr), 
\]
and further applying Lemma~\ref{lem:theta-bound} gives the desired 
result~\eqref{eqn:ABDA-rate-no-min}.
\end{proof}

As a sanity check, we show that the right-hand-side 
of~\eqref{eqn:ABDA-rate-no-min} is strictly positive for any $x\in\dom h$
such that $ F(x)< F(z_1)+L D_h(x, z_1)$.
We exploit the fact that $z_1=\argmin_{z\in C} \{ \ell(z|z_0) + L h(z)\}$. 
Using Lemma~\ref{lem:basic-lemma-h}, we have 
\[
    \ell(z_1|z_0) + L h(z_1) \leq \ell(x|z_0) + L h(x) - L D_h(x, z_1),
\]
which implies
\[
    L\bigl(h(x)-h(z_1)\bigr) \geq L D_h(x,z_1) + \ell(z_1|z_0)-\ell(x|z_0).
\]
Then we have
\begin{eqnarray*}
L\bigl(h(x)-h(z_1)+D_h(z_1,z_0)\bigr)
&\geq& L D_h(x,z_1) + \ell(z_1|z_0)-\ell(x|z_0) + L D_h(z_1, z_0)\\
&=& L D_h(x,z_1) +\bigl(\ell(z_1|z_0)+L D_h(z_1,z_0)\bigr) - \ell(x|z_0)\\
&\geq& L D_h(x,z_1) +  F(z_1) - \ell(x|z_0)\\
&\geq& L D_h(x,z_1) +  F(z_1) -  F(x),
\end{eqnarray*}
where the second inequality used the upper bound in~\eqref{eqn:sanwitch},
and the last inequality used the lower bound in~\eqref{eqn:sanwitch}. 
Therefore, for any $x$ such that $ F(x)< F(z_1)+L D_h(x,z_1)$, we have
\[
    L\bigl(h(x)-h(z_1)+D_h(z_1,z_0)\bigr) ~>~ L D_h(x,z_1) ~\geq~ 0.
\]
This completes the proof.

\section{Numerical experiments}
\label{sec:experiments}

We consider three applications of relatively smooth convex optimization:
D-optimal experiment design, Poisson linear inverse problem, 
and relative-entropy nonnegative regression.
For each application, we compare the algorithms developed in this paper
with the BPG method~\eqref{eqn:BPG} and demonstrate significant 
performance improvement. 
Our implementations and experiments are shared through an open-source
repository at \url{https://github.com/linxiaolx/accbpg}.

\subsection{D-optimal experiment design}
\label{sec:D-optimal}

\begin{figure}[t]
\begin{subfigure}[b]{0.5\linewidth}
  \centering
  \includegraphics[width=0.9\textwidth]{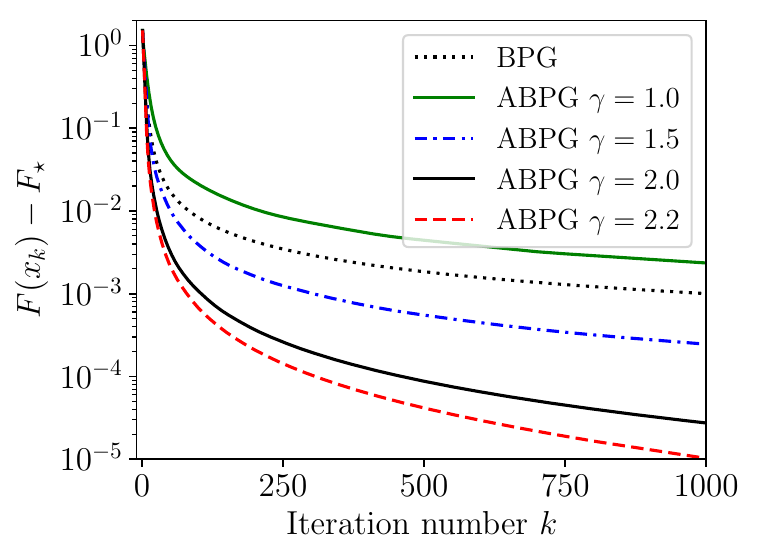}
  \caption{ABPG method with different TSE $\gamma$.}
  \label{fig:D-opt-gamma-obj}
\end{subfigure}%
\begin{subfigure}[b]{0.5\linewidth}
  \centering
  \includegraphics[width=0.9\textwidth]{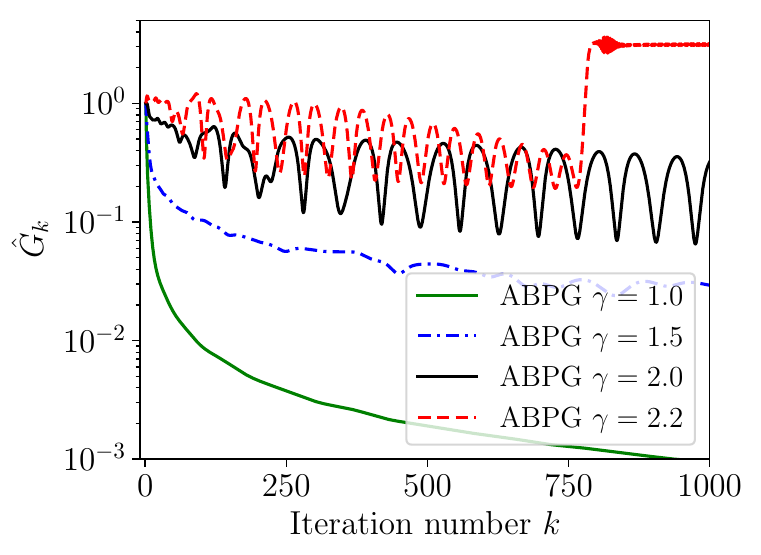}
  \caption{Local triangle scaling gain $\widehat{G}_k$.}
  \label{fig:D-opt-gamma-gain}
\end{subfigure}
\caption{D-optimal design: random problem instance with $m=80$ and $n=200$.}
\label{fig:D-opt-sublinear}
\end{figure}

Given $n$ vectors $v_1,\ldots,v_n\in\reals^m$ where $n\geq m+1$, 
the D-optimal design problem is
\begin{equation}\label{eqn:D-optimal}
\begin{array}{ll}
\minimize & f(x):=-\log\det\left(\sum_{i=1}^n x^{(i)} v_i v_i^T\right)\\[1ex]
\mbox{subject to} & \sum_{i=1}^n x^{(i)}  = 1 \\[1ex]
                  & x^{(i)} \geq 0, \quad i=1,\ldots,n.
\end{array}
\end{equation}
In the form of problem~\eqref{eqn:composite-min}, we have $\Psi\equiv 0$,
$ F(x)\equiv f(x)$, and $C$ is the standard simplex in~$\reals^n$.
In statistics, this problem corresponds to maximizing the determinant of the
Fisher information matrix (e.g., \cite{KieferWolfowitz59,Atwood69}). 
It is shown in \cite{LuFreundNesterov18} that~$f$ 
defined in~\eqref{eqn:D-optimal} is $1$-smooth relative to Burg's entropy
$h(x)=-\sum_{i=1}^n\log(x^{(i)})$ on $\reals^n_+$.
In this case, $D_h$ is the IS-distance defined in~\eqref{eqn:IS-distance}.

\subsubsection{Experiment on synthetic data}
In our first experiment, we set $m=80$ and $n=200$ and generated $n$ random
vectors in $\reals^m$, where the entries of the vectors were generated
following independent Gaussian distributions with zero mean and unit variance.
The results are shown in Figures~\ref{fig:D-opt-sublinear} and~\ref{fig:D-opt-sublinear-time}.

\begin{figure}[t]
\begin{subfigure}[b]{0.5\linewidth}
  \centering
  \includegraphics[width=0.9\textwidth]{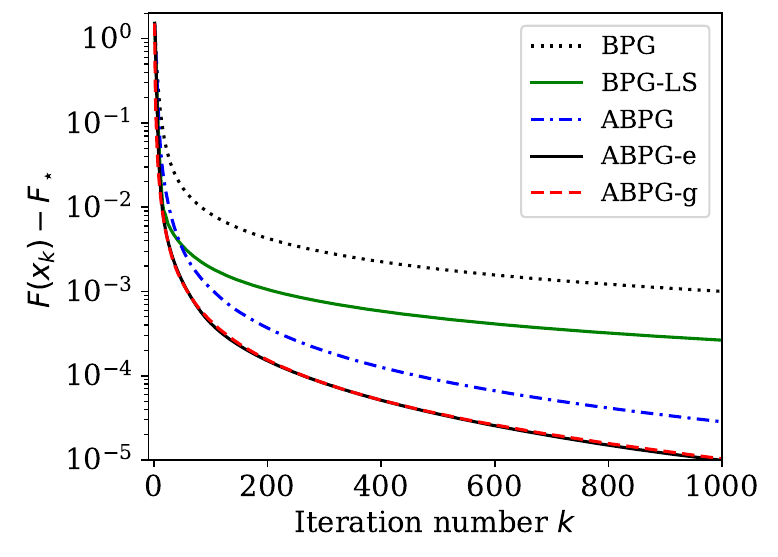}
  \caption{Objective gap versus iterations (linear scale).}
\vspace{2ex}
  \label{fig:D-opt-adapt-lin}
\end{subfigure}%
\begin{subfigure}[b]{0.5\linewidth}
  \centering
  \includegraphics[width=0.9\textwidth]{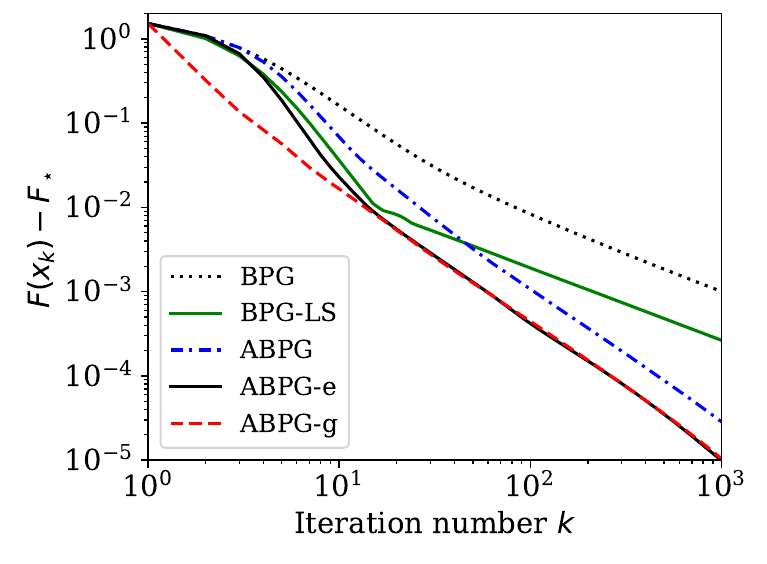}
  \caption{Objective gap versus iterations (log scale).}
\vspace{2ex}
  \label{fig:D-opt-adapt-log}
\end{subfigure}
\begin{subfigure}[b]{0.5\linewidth}
  \centering
  \includegraphics[width=0.9\textwidth]{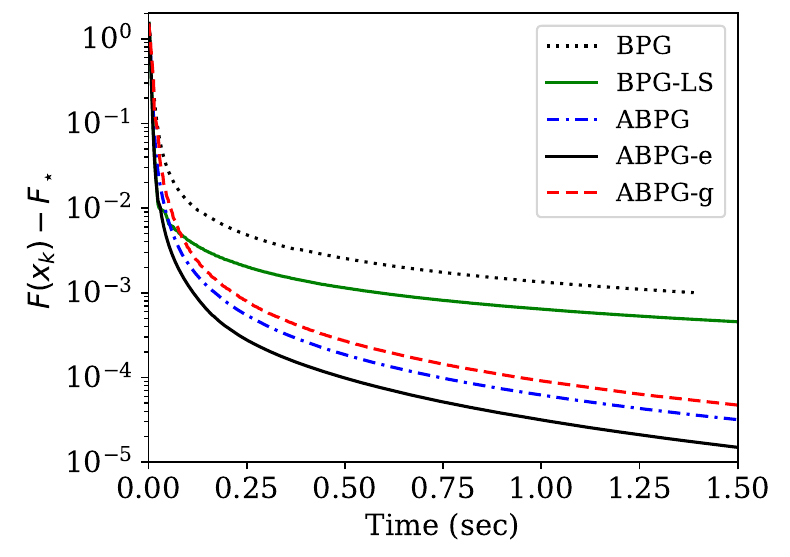}
  \caption{Objective gap versus CPU time (linear scale).}
  \label{fig:D-opt-adapt-lin-time}
\end{subfigure}%
\begin{subfigure}[b]{0.5\linewidth}
  \centering
  \includegraphics[width=0.9\textwidth]{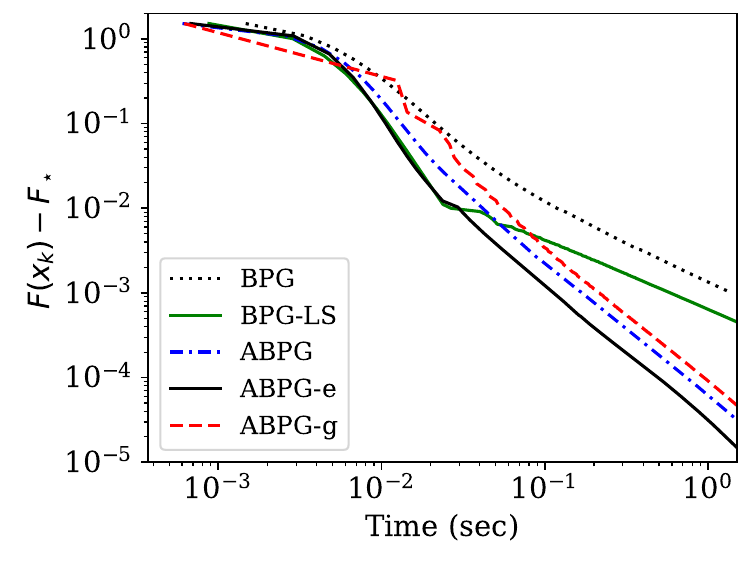}
  \caption{Objective gap versus CPU time (log scale).}
  \label{fig:D-opt-adapt-log-time}
\end{subfigure}
\caption{D-optimal design: random instance with $m=80$ and $n=200$; $\gamma=2$ for all ABPG variants.}
\label{fig:D-opt-sublinear-time}
\end{figure}

Figure~\ref{fig:D-opt-gamma-obj} shows the reduction of optimality gap by 
the BPG method~\eqref{eqn:BPG} and the ABPG method (Algorithm~\ref{alg:ABPG})
with four different values of~$\gamma$.
For $\gamma=1$, the ABPG method converges with $O(k^{-1})$ rate, 
but is slower than the BPG method.
When we increase~$\gamma$ to~$1.5$ and then~$2$, the ABPG method is 
significantly faster than BPG.
Interestingly, ABPG still converges with $\gamma=2.2$ 
(which is larger than the intrinsic TSE $\gammain=2$)
and is even faster than with $\gamma=2$.
To better understand this phenomenon, we plot the 
local triangle-scaling gain
\begin{equation}\label{eqn:effective-gain}
  \widehat{G}_k
=\frac{D_h(x_{k+1},y_k)}{\theta^\gamma D_h(z_{k+1},z_k)}
=\frac{D_h((1-\theta)x_k+\theta z_{k+1},(1-\theta)x_k+\theta z_k)}
  {\theta^\gamma D_h(z_{k+1},z_k)}.
\end{equation}
Figure~\ref{fig:D-opt-gamma-gain} shows that for $\gamma=1.0$ and $1.5$,
$\widehat{G}_k$ is mostly much smaller than~$1$.
For $\gamma=2$, $\widehat{G}_k$ is much closer to~$1$ but always less than~$1$.
This gives a numerical certificate that the ABPG method converged with 
$O(k^{-2})$ rate.
For $\gamma=2.2$, $\widehat{G}_k$ stayed close to~$1$ for the first 700 
iterations and then jumped to~$3$ and stayed around.
The method diverges with larger value of~$\gamma$.
We didn't plot the ABDA method (Algorithm~\ref{alg:ABDA}) because it overlaps
with ABPG for the same value of~$\gamma$ when the initial point is taken
as the center of the simplex, see part (a) of Theorem~\ref{thm:ABDA-rate}.

\begin{figure}[t]
  \centering
  \includegraphics[width=0.45\textwidth]{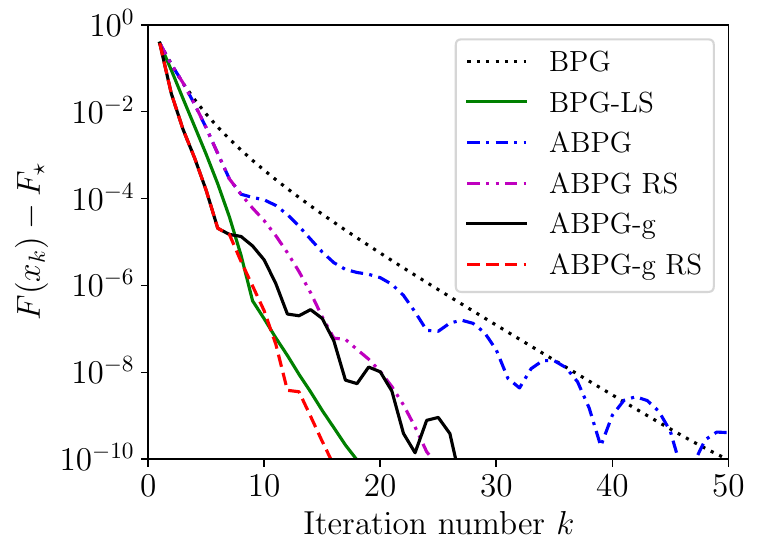}
  \hspace{1cm}
  \includegraphics[width=0.45\textwidth]{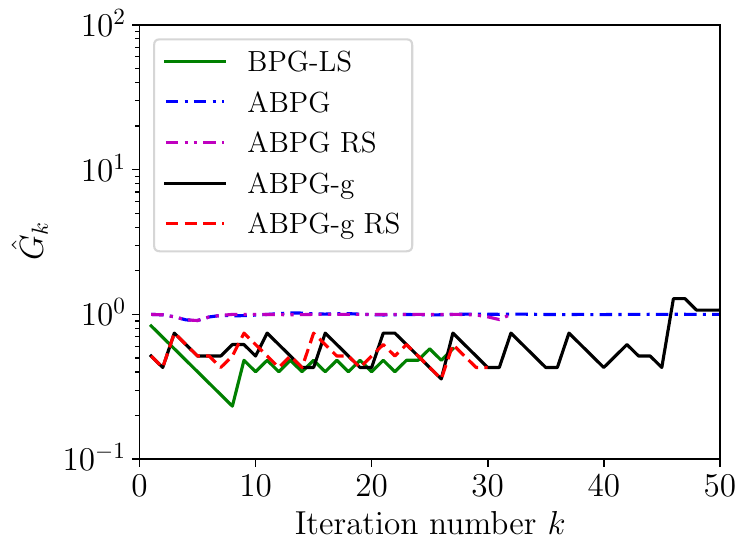}
  \caption{D-optimal design: random problem instance with $m=80$ and $n=120$.}
  \label{fig:D-opt-restart}
\end{figure}

Figure~\ref{fig:D-opt-adapt-lin}
compares the basic BPG and ABPG methods with their adaptive variants. 
The BPG-LS method is a variant of BPG equipped with the same adaptive 
line-search scheme in Algorithm~\ref{alg:ABPG-g} 
(see also \cite[Method~3.3]{Nesterov13composite}).
For all variants of ABPG, we set $\gamma=\gammain=2$.
For BPG-LS and ABPG-g, we set $\rho=1.5$ for adjusting the gain $G_k$.
The adaptive variants converged in fewer iterations than their respective basic versions.
Figure~\ref{fig:D-opt-adapt-log} shows the same results in log-log scale.
We can clearly see the different slopes of the BPG variants and ABPG 
variants, demonstrating their $O(k^{-1})$ and $O(k^{-2})$ convergence rates
respectively.
For ABPG-e, we started with $\gamma_0=3$ and it eventually settled down to
$\gamma=2$, which is reflected in its gradual change of slope in 
Figure~\ref{fig:D-opt-adapt-log}.

We also show the comparison in terms of CPU time in 
Figures~\ref{fig:D-opt-adapt-lin-time} and \ref{fig:D-opt-adapt-log-time}.
As remarked at the end of Section~\ref{sec:ABPG-expo}, the ABPG-e method only take a constant number more iterations than ABPG, thus its their comparison is very similar to the case with number of iterations.
For ABPG-g, the analysis in Section~\ref{sec:ABPG-gain} on page~\pageref{remark:ABPG-e-oracles} shows that the number of gradient calls and proximal computations is
roughly twice of the ABPG method with the same number of iterations.
This is exactly what we observe in Figures~\ref{fig:D-opt-adapt-lin-time} and \ref{fig:D-opt-adapt-log-time}.
Given such predictable scaling between number of iterations and CPU time, we only show comparisons in the number of iterations in the rest numerical experiments.

Figure~\ref{fig:D-opt-restart} shows the comparison of different methods
on another random problem instance with $m=80$ and $n=120$.
All methods converge much faster and reach very high precision.
In particular, BPG and BPG-LS look to have linear convergence.
This indicates that this problem instance is much better conditioned and the 
objective function may be strongly convex relative to Burg's entropy.
In this case, it is shown in \cite{LuFreundNesterov18} that the BPG method
attains linear convergence.
The ABPG and ABPG-g methods demonstrate periodic non-monotone behavior.
A well-known technique to avoid such oscillations and attain fast linear
convergence is to restart the algorithm whenever the function value starts
to increase \cite{ODonoghueCandes15restart}.
We applied restart (RS) to both ABPG and ABPG-g, which resulted in a much faster
convergence as shown in Figure~\ref{fig:D-opt-restart}.


\begin{figure}[p]
\begin{subfigure}[b]{0.33\linewidth}
  \centering
  \includegraphics[width=1.00\textwidth]{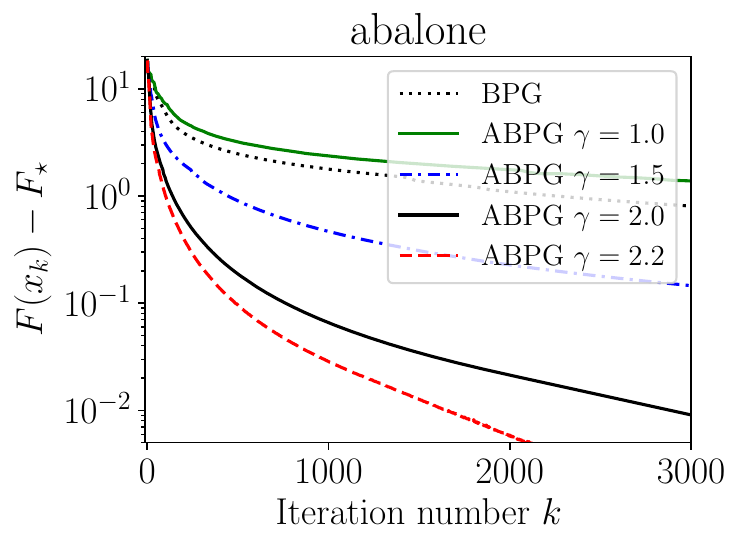}
  \label{fig:D-opt-gamma-obj-abalone}
\end{subfigure}%
\begin{subfigure}[b]{0.33\linewidth}
  \centering
  \includegraphics[width=1.00\textwidth]{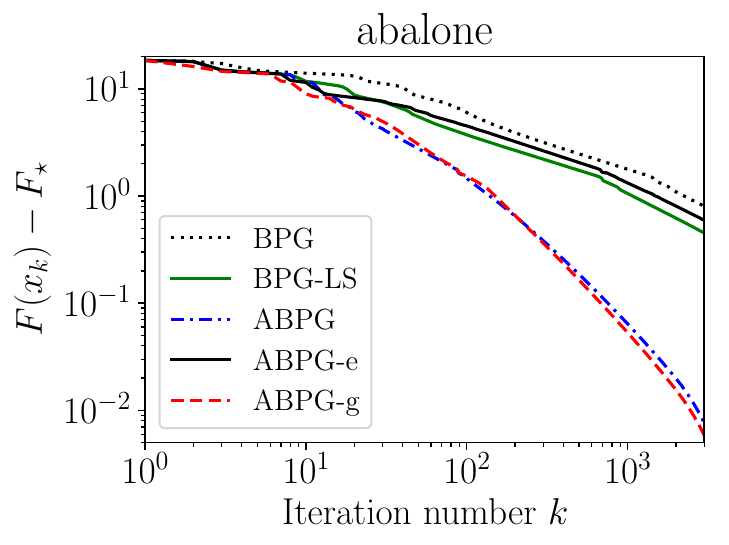}
  \label{fig:D-opt-adapt-obj-abalone}
\end{subfigure}%
\begin{subfigure}[b]{0.33\linewidth}
  \centering
  \includegraphics[width=1.00\textwidth]{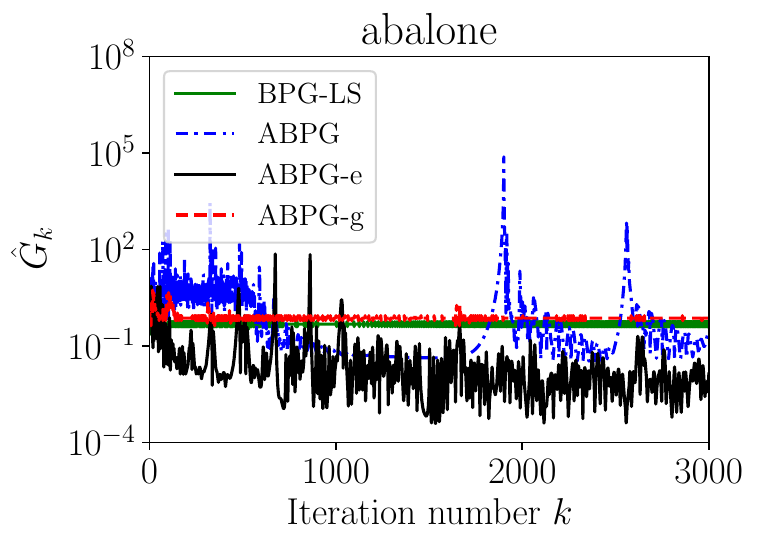}
  \label{fig:D-opt-adapt-gain-abalone}
\end{subfigure}

\begin{subfigure}[b]{0.33\linewidth}
  \centering
  \includegraphics[width=1.0\textwidth]{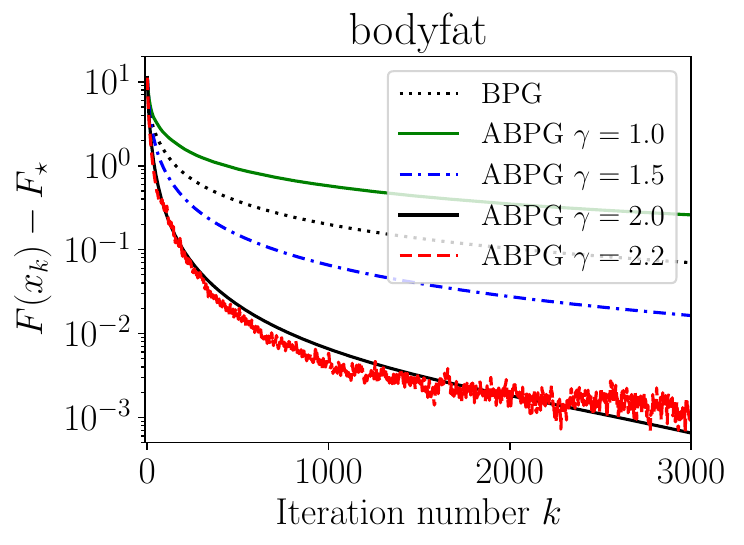}
  \label{fig:D-opt-gamma-obj-bodyfat}
\end{subfigure}%
\begin{subfigure}[b]{0.33\linewidth}
  \centering
  \includegraphics[width=1.0\textwidth]{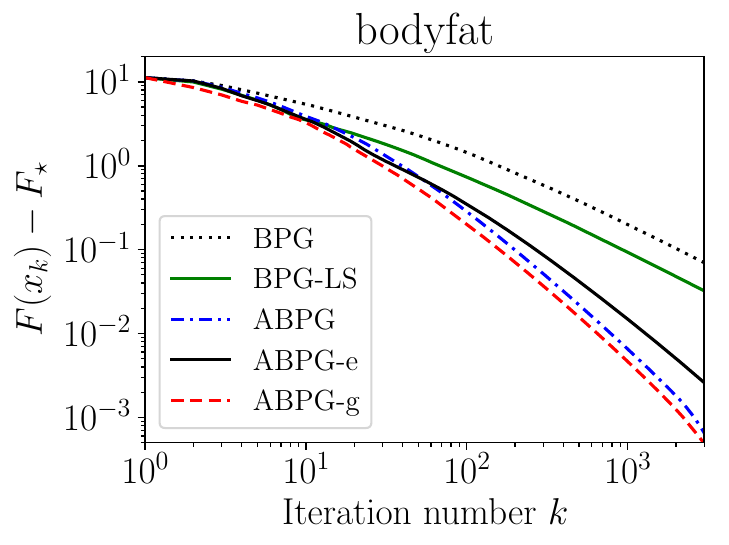}
  \label{fig:D-opt-adapt-obj-bodyfat}
\end{subfigure}%
\begin{subfigure}[b]{0.33\linewidth}
  \centering
  \includegraphics[width=1.0\textwidth]{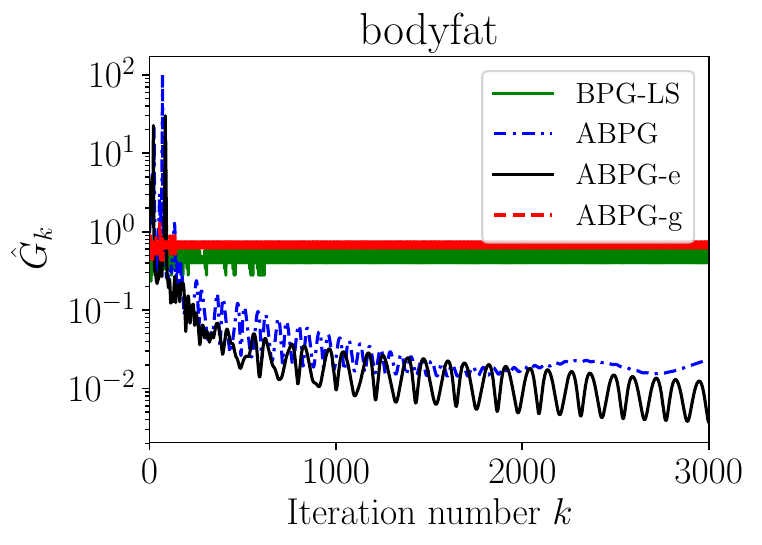}
  \label{fig:D-opt-adapt-gain-bodyfat}
\end{subfigure}


\begin{subfigure}[b]{0.33\linewidth}
  \centering
  \includegraphics[width=1.00\textwidth]{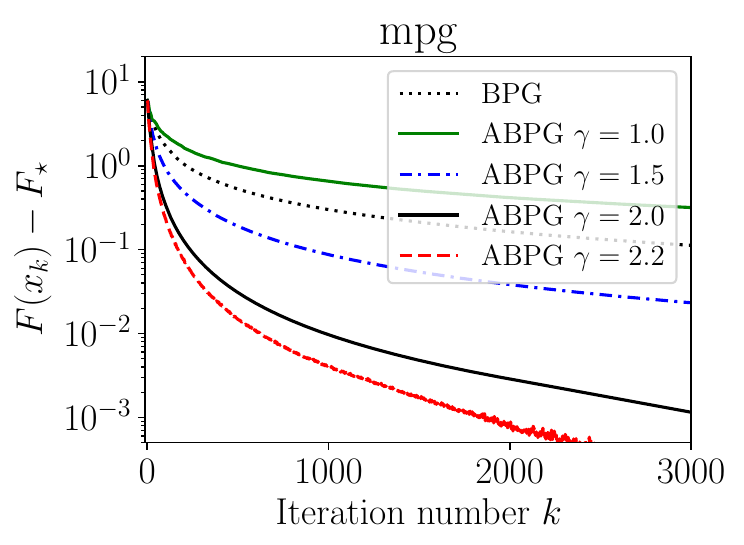}
  \label{fig:D-opt-gamma-obj-mpg}
\end{subfigure}%
\begin{subfigure}[b]{0.33\linewidth}
  \centering
  \includegraphics[width=1.00\textwidth]{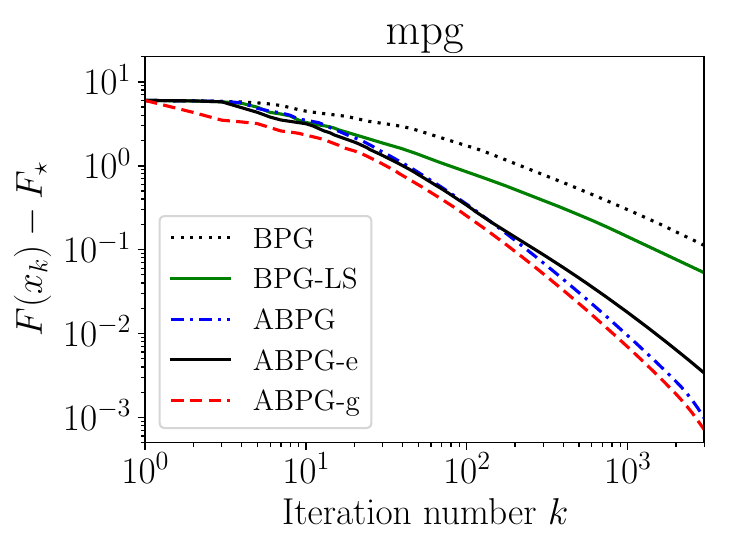}
  \label{fig:D-opt-adapt-obj-mpg}
\end{subfigure}%
\begin{subfigure}[b]{0.33\linewidth}
  \centering
  \includegraphics[width=1.00\textwidth]{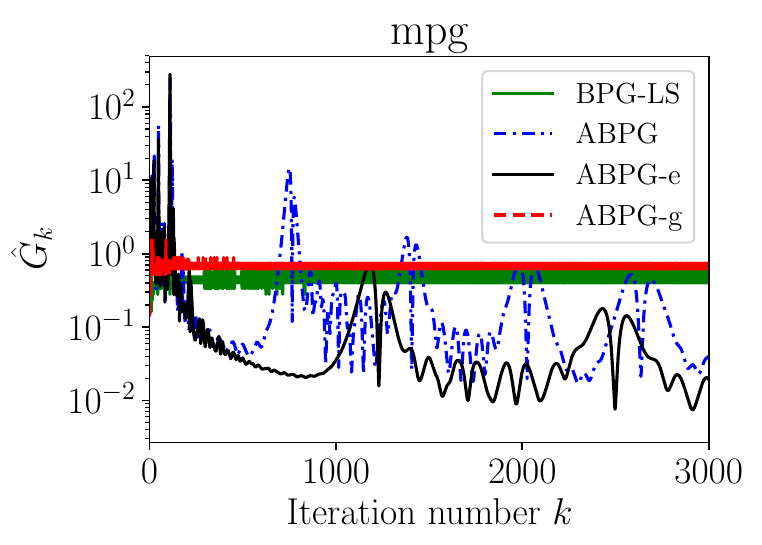}
  \label{fig:D-opt-adapt-gain-mpg}
\end{subfigure}

\begin{subfigure}[b]{0.33\linewidth}
  \centering
  \includegraphics[width=1.00\textwidth]{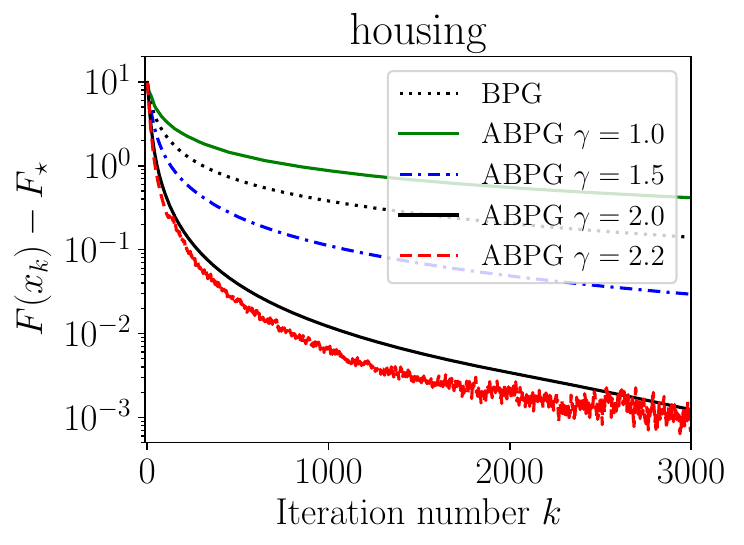}
  \label{fig:D-opt-gamma-obj-housing}
\end{subfigure}%
\begin{subfigure}[b]{0.33\linewidth}
  \centering
  \includegraphics[width=1.00\textwidth]{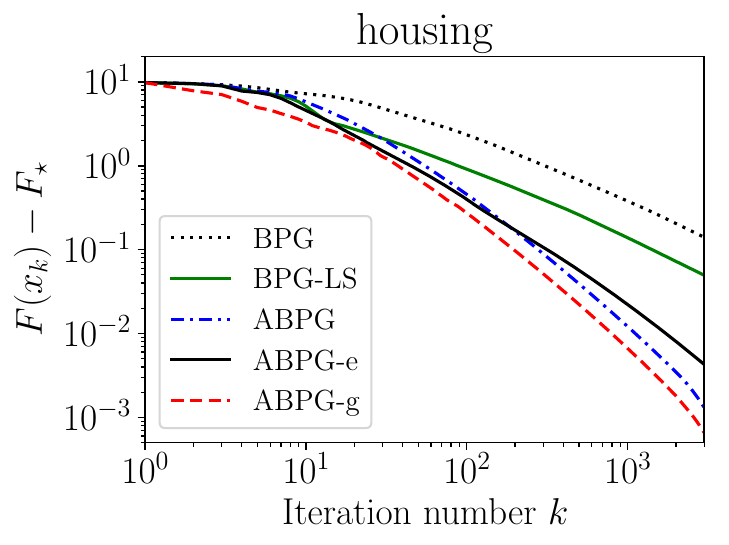}
  \label{fig:D-opt-adapt-obj-housing}
\end{subfigure}%
\begin{subfigure}[b]{0.33\linewidth}
  \centering
  \includegraphics[width=1.00\textwidth]{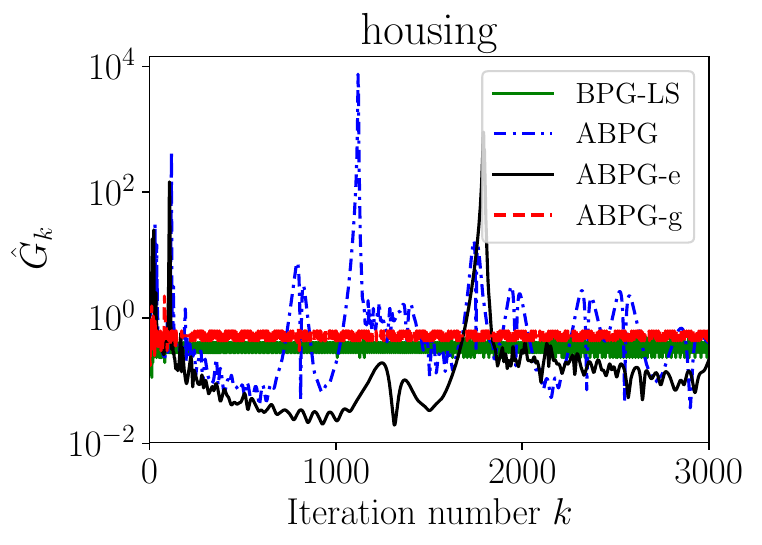}
  \label{fig:D-opt-adapt-gain-housing}
\end{subfigure}
\vspace{-1ex}
\caption{D-optimal design on several LibSVM datasets: 
\texttt{abalone}, 
\texttt{bodyfat}, 
\texttt{mpg}      
and \texttt{housing}.
The first column compares the BPG method, ABPG method with $\gamma\in\{1.0,1.5,2.0,2.2\}$ and the ABDA method with $\gamma=2.0$, plotted in log-linear scale.
The second column compares BPG, BPG-LS and three ABPG variants with $\gamma=2$, plotted in log-log scale, and the third column plots the corresponding local
triangle-scaling gains. 
For ABPG and ABPG-e, these are the $\hat{G}_k$ estimated using~\eqref{eqn:effective-gain};
For BPG-LS and ABPG-g, they are calculated as part of the gain adaptation schemes.
}
\label{fig:D-opt-sublinear_libsvm}
\end{figure}

\subsubsection{Experiment on real data}
In our second experiment, we construct D-optimal design instances from LibSVM data~\cite{chang2011libsvm}. In particular, we consider several regression datasets -- the goal is to find the most relevant data points where one shall run the experiment to evaluate the corresponding label.

Figure~\ref{fig:D-opt-sublinear_libsvm} shows the results on four different datasets: 
\texttt{abalone} ($n=4177, m=8$),  
\texttt{bodyfat} ($n=252, m=14$), 
\texttt{mpg} ($n=392, m=7$) and 
\texttt{housing} ($n=506, m=13$). 
The left column indicates that in each case, the best performance of ABPG is achieved with large TSE $\gamma = 2$ and $\gamma=2.2$. Furthermore, ABPG with $\gamma > 1$ always compared favorably over plain BPG. 
 
Next, the second column of Figure~\ref{fig:D-opt-sublinear_libsvm} shows that both ABPG-g and ABPG (with $\gamma=2$) always significantly outperform BPG and BPG-LS. We have chosen log-log scale of the plot to contrast the $O(k^{-1})$ convergence rate of BPG (with line search) with the $O(k^{-2})$ convergence rate of ABPG and ABPG-g. 
In the third column, we plot the local triangle-scaling gains. They serve as numerical certificates of the empirical $O(k^{-2})$ convergence rate of ABPG and its variants.
In particular, we see that $G_k$ for the ABPG-g algorithm is mostly flat and less than one.

\subsection{Poisson linear inverse problem}
\label{sec:Poisson}

\begin{figure}[t]
\begin{subfigure}[b]{0.5\linewidth}
  \centering
  \includegraphics[width=0.9\textwidth]{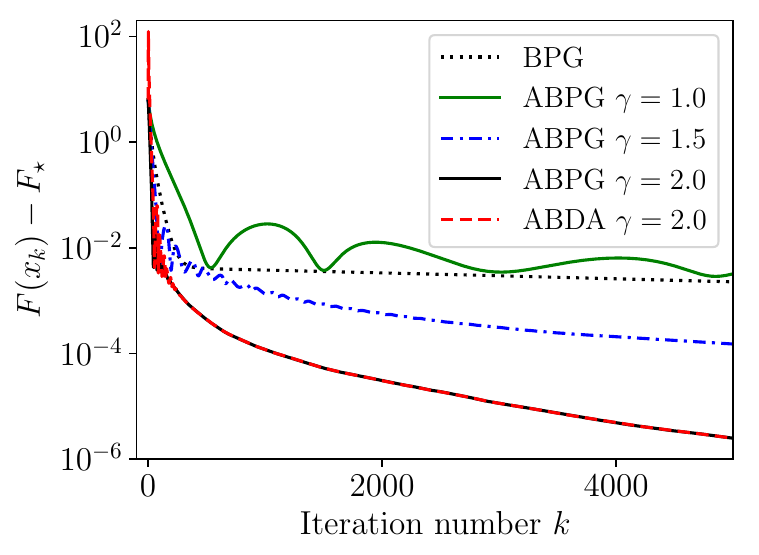}
  \caption{ABPG (varying $\gamma$) and ABDA ($\gamma=2$) methods.}
\vspace{2ex}
  \label{fig:Poisson-gamma-obj}
\end{subfigure}%
\begin{subfigure}[b]{0.5\linewidth}
  \centering
  \includegraphics[width=0.9\textwidth]{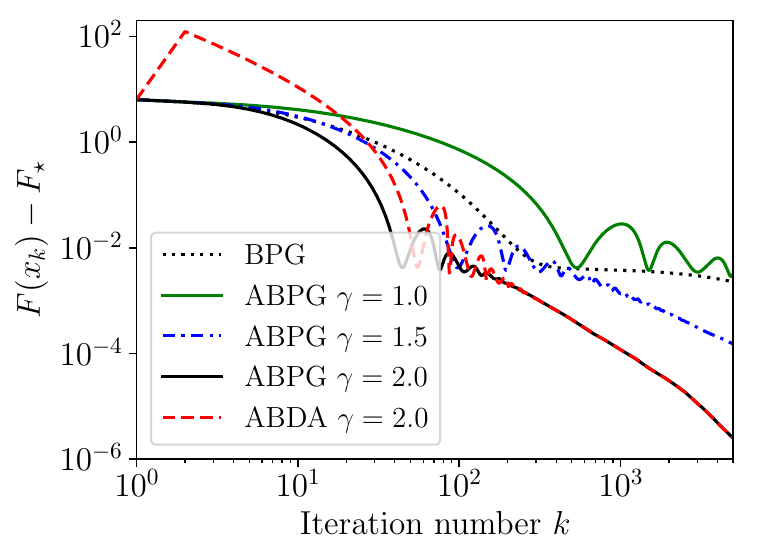}
  \caption{Same results in (a) in log-log plot.}
\vspace{2ex}
  \label{fig:Poisson-gamma-log}
\end{subfigure}
\begin{subfigure}[b]{0.5\linewidth}
  \centering
  \includegraphics[width=0.9\textwidth]{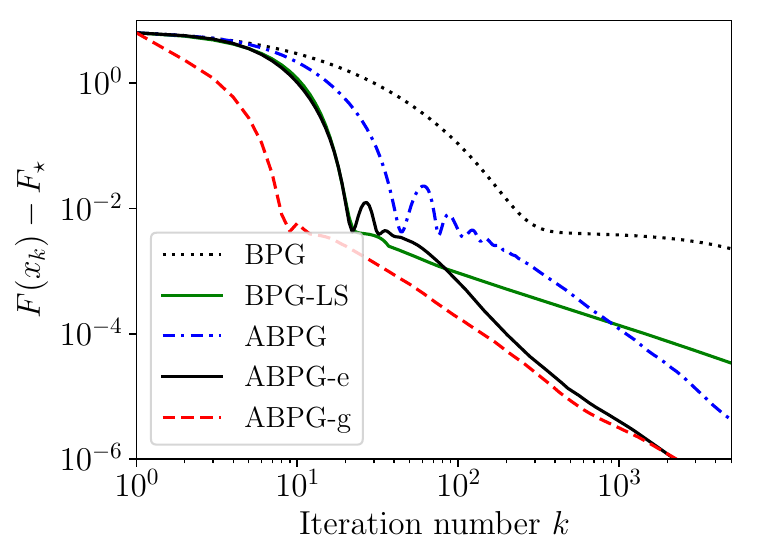}
  \caption{Adaptive ABPG methods $\gamma=2$ (log-log plot).}
  \label{fig:Poisson-adapt-log}
\vspace{1ex}
\end{subfigure}%
\begin{subfigure}[b]{0.5\linewidth}
  \centering
  \includegraphics[width=0.9\textwidth]{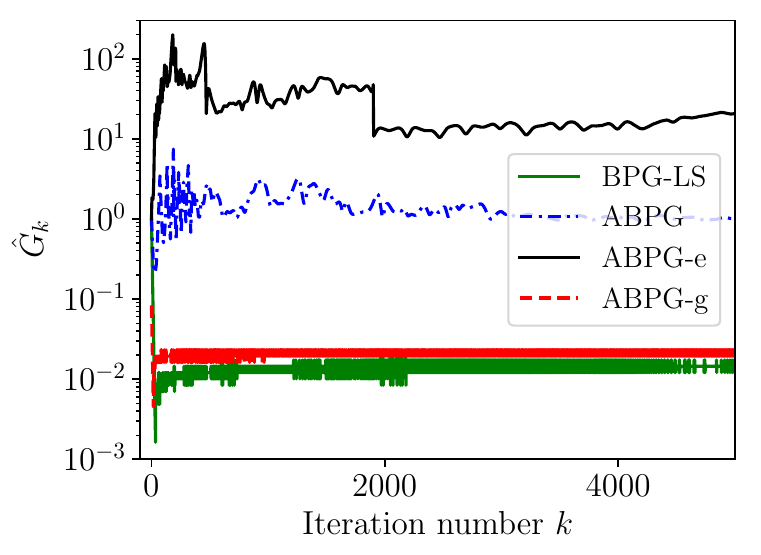}
  \caption{Local triangle-scaling gain $\widehat{G}_k$.}
  \label{fig:Poisson-adapt-gain}
\vspace{1ex}
\end{subfigure}
\caption{Poisson linear inverse problem: random instance with $m=200$ and $n=100$.}
\label{fig:Poisson-sublinear}
\end{figure}

\begin{figure}[t]
  \centering
  \includegraphics[width=0.45\textwidth]{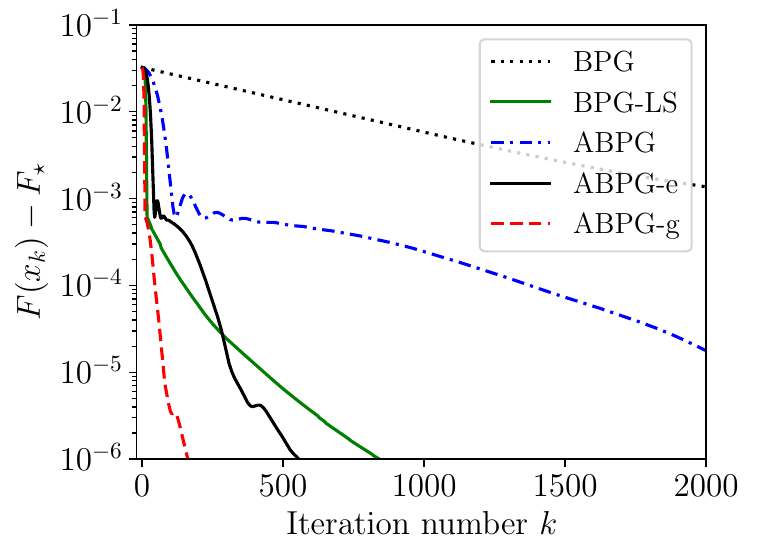}
  \hspace{0.5cm}
  \includegraphics[width=0.45\textwidth]{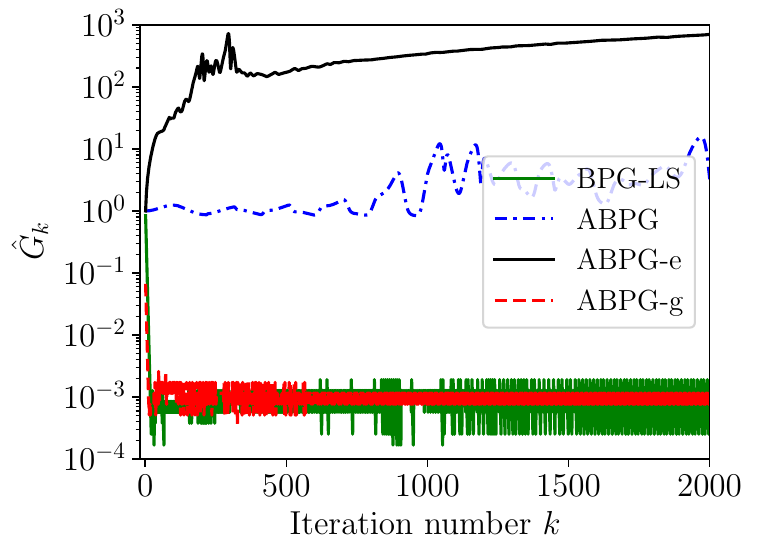}
  \caption{Poisson linear inverse problem: random instance with $m=100$ and
 $n=1000$.}
  \label{fig:Poisson-large-n}
\end{figure}

In Poisson inverse problems (e.g., \cite{Csiszar91,Bertero09}), 
we are given a nonnegative observation matrix $A\in\reals^{m\times n}_+$
and a noisy measurement vector $b\in\reals^m_{++}$, and the goal is to
reconstruct the signal $x\in\reals^n_+$ such that $Ax\approx b$.
A natural measure of closeness of two nonnegative vectors is the KL-divergence
defined in~\eqref{eqn:KL-divergence}. 
In particular, minimizing $\DKL(b,Ax)$ corresponds to maximizing the Poisson
log-likelihood function. 
We consider problems of the form
\[
\minimize_{x\in\reals^n_+} ~ F(x):=\DKL(b,Ax) + \Psi(x),
\]
where $\Psi(x)$ is a simple regularization function.
It is shown in \cite{BauschkeBolteTeboulle17} that the function
$f(x)=\DKL(b,Ax)$ is $L$-smooth relative to
$h(x)=-\sum_{i=1}^n \log (x^{(i)})$ on $\reals^n_+$
for any $L\geq\|b\|_1=\sum_{i=1}^m b^{(i)}$.
Therefore, in the BPG and ABPG methods, we use again the IS-distance $\DIS$
defined in~\eqref{eqn:IS-distance} as the proximity measure.

Figure~\ref{fig:Poisson-sublinear} shows our computational results for a
randomly generated instance with $m=200$ and $n=100$ and $\Psi\equiv 0$
(no regularization).
The entries of~$A$ and~$b$ are generated following independent uniform 
distribution over the interval $[0, 1]$.

Figure~\ref{fig:Poisson-gamma-obj} shows the reduction of objective gap
by BPG and ABPG with $\gamma=1.0$, $1.5$ and $2.0$, as well as the ABDA
method (Algorithm~\ref{alg:ABDA}).
ABPG and ABDA with $\gamma=2$ mostly overlap each other in this figure.
Figure~\ref{fig:Poisson-gamma-log} plots the same results in log-log scale,
which reveals that ABPG and ABDA (both with $\gamma=2$) behave quite 
differently in the beginning.
The ABDA method has a jump of objective value at $k=1$ because
$z_0\neq\argmin_{z\in C}h(z)$, and its convergence rate is governed by
part~(b) of Theorem~\ref{thm:ABDA-rate}.
In fact, for $C=\reals^n_{+}$, Burg's entropy 
$h(x)=-\sum_{i=1}^n \log (x^{(i)})$ is unbounded below as $\|x\|\to\infty$.
In contrast, for the D-optimal design problem in Section~\ref{sec:D-optimal},
$C$ is the standard simplex, and if we choose $z_0=x_0=(1/n,\ldots,1/n)$
then $z_0=\argmin_{z\in C} h(z)$.
In that case, we can show that ABPG and ABDA are equivalent when $\Psi\equiv 0$.

Figure~\ref{fig:Poisson-adapt-log} compares the basic and adaptive variants
of BPG and ABPG. 
For the ABPG and ABPG-g methods, we set $\gamma=\gammain=2$.
For ABPG-e, we start with $\gamma_0=3$, and the final $\gamma_k=2.8$ after
$k=5000$ iterations ($\delta=0.2$ in Algorithm~\ref{alg:ABPG-e}).
Although ABPG-e uses a much larger~$\gamma$ most of the time,
we see ABPG-g converges faster than ABPG-e in the beginning and they 
eventually become similar. 
This can be explained through the effective triangle-scaling gains plotted
in Figure~\ref{fig:Poisson-adapt-gain}.
For ABPG and ABPG-e, the effective gains plotted are $\widehat{G}_k$
defined in~\eqref{eqn:effective-gain}.
For BPG-LS and ABPG-g, we plot the $G_k$'s which are adjusted directly in 
the algorithms.
For ABPG-g, $G_k\approx 0.025$ most of the time.
The effective $\widehat{G}_k$ for ABPG-e is almost $1000$ times larger,
which counters the large value of~$\gamma$ used.
The sudden reduction of $\widehat{G}_k$ around $k=2000$ is when
$\gamma$ is reduced from~$3$ to $2.8$.
We expect $\gamma_k\to2$ as $k$ continues to increase.

Figure~\ref{fig:Poisson-large-n} shows the results for a randomly generated
instance with $m=100$ and $n=1000$. 
In this case, since $m<n$, we added a regularization 
$\Psi(x)=(\lambda/2)\|x\|^2$ with $\lambda=0.001$.
ABPG-g has the best performance. Again we observe that $G_k\ll 1$ most of the 
time, which gives a numerical certificate that the ABPG methods do converge 
with $O(k^{-2})$ rate.

\subsection{Relative-entropy nonnegative regression}

An alternative approach for solving the nonnegative linear inverse problem 
described in Section~\ref{sec:Poisson} is to minimize $\DKL(Ax,b)$, i.e.,
\[
\minimize_{x\in\reals^n_+} ~ F(x):=\DKL(Ax,b) + \Psi(x).
\]
In this case, 
it is shown in \cite{BauschkeBolteTeboulle17} that $f(x)=\DKL(Ax,b)$ is
$L$-smooth relative to the Boltzmann-Shannon entropy
$h(x)=\sum_{i=1}^n x^{(i)}\log (x^{(i)})$ on $\reals^n_+$
for any $L$ such that 
\[
L ~\geq~ \max_{1\leq j\leq n}\sum_{i=1}^m A_{ij}
~=~\max_{1\leq j\leq n}\|A_{:j}\|_1
\]
where $A_{:j}$ denotes the $j$th column of~$A$.
Therefore, in the BPG and ABPG methods, we use the KL-divergence $\DKL$
defined in~\eqref{eqn:KL-divergence} as the proximity measure.
In our experiment, we apply $\ell_1$-regularization
$\Psi(x)=\lambda\|x\|_1$ with $\lambda=0.001$.

Figure~\ref{fig:Nonneg-large-m} shows the results for a randomly generated
instance with $m=1000$ and $n=100$.
For all variants of the ABPG method, we set $\gamma=\gammain=2$.
Figure~\ref{fig:Nonneg-large-n} shows the results for a random instance
with $m=100$ and $n=1000$.
In this case, we clearly see linear convergence of the BPG and BPG-LS methods.
Since the accelerated methods demonstrate oscillations in objective value,
we tried the restart (RS) trick \cite{ODonoghueCandes15restart} 
and obtained faster convergence with apparent linear rate.
In contrast, ABPG methods with restart do not make any difference 
in~Figure~\ref{fig:Nonneg-large-m}.
Although not shown here, 
for the ABPG-g method, we always obtain small gains $G_k\leq 1$ at each step.
Therefore their geometric mean $\overline{G}_k$ is also small, which serves
as a certificate of the $O(k^{-2})$ convergence rate for this problem instance.

\begin{figure}[t]
\begin{subfigure}[b]{0.5\linewidth}
  \centering
  \includegraphics[width=0.9\textwidth]{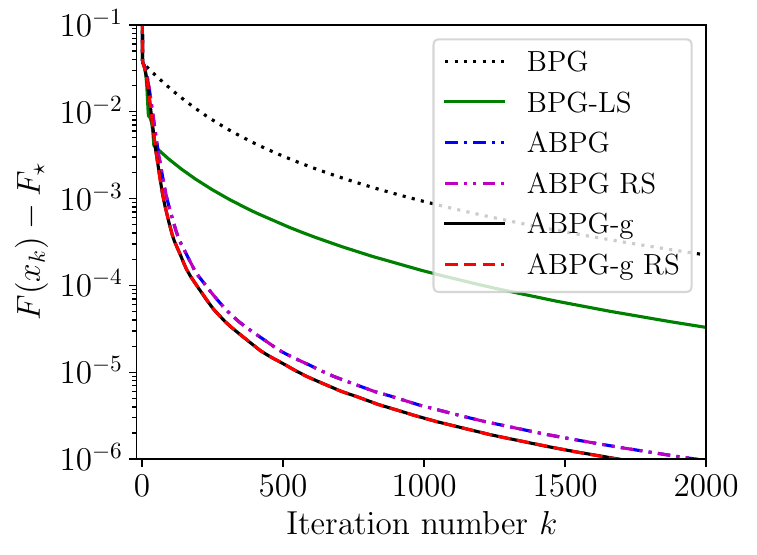}
  \caption{$m=1000$ and $n=100$.}
  \label{fig:Nonneg-large-m}
\end{subfigure}%
\begin{subfigure}[b]{0.5\linewidth}
  \centering
  \includegraphics[width=0.9\textwidth]{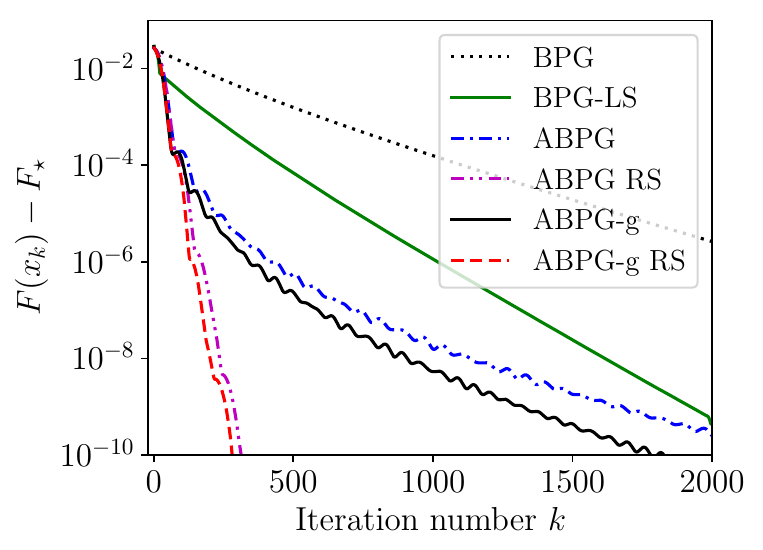}
  \caption{$m=100$ and $n=1000$.}
  \label{fig:Nonneg-large-n}
\end{subfigure}
  \centering
  \caption{Two random instances of relative entropy nonnegative regression
           ($\gamma=2$ for ABPG).}
  \label{fig:KL-regr-restart}
\end{figure}

\section*{Acknowledgments}
We thank Haihao Lu, Robert Freund and Yurii Nesterov for helpful conversations. 
Peter Richt{\'a}rik acknowledges the support of the KAUST Baseline Research Funding Scheme.

\bibliographystyle{abbrv}
\bibliography{ABPG}

\end{document}